\documentclass[final]{siamltex}

\usepackage{amsmath, amssymb, mathtools, extarrows, cite, multirow,extarrows}
\usepackage{tikz-cd}
\allowdisplaybreaks

\newcommand{\mzero}{\mathbf O }
\newcommand{\vzero}{\mathbf 0 }
\newcommand{\hatdSmR}[2]{\widehat{\mathbf d}^{#1}}
\newcommand{\cond}[2]{\textup{({#1}{\scriptsize {#2}})}}

\newtheorem{remark}[theorem]{Remark}
\newtheorem{example}[theorem]{Example}

\makeatletter
\newif\if@borderstar
\def\bordermatrix{\@ifnextchar*{%
\@borderstartrue\@bordermatrix@i}{\@borderstarfalse\@bordermatrix@i*}%
}
 \def\@bordermatrix@i*{\@ifnextchar[{\@bordermatrix@ii}{\@bordermatrix@ii[()
]}}
 \def\@bordermatrix@ii[#1]#2{%
 \begingroup
 \m@th\@tempdima8.75\p@\setbox\z@\vbox{%
 \def\cr{\crcr\noalign{\kern 2\p@\global\let\cr\endline }}%
 \ialign {$##$\hfil\kern 2\p@\kern\@tempdima & \thinspace %
 \hfil $##$\hfil && \quad\hfil $##$\hfil\crcr\omit\strut %
 \hfil\crcr\noalign{\kern -\baselineskip}#2\crcr\omit %
 \strut\cr}}%
 \setbox\tw@\vbox{\unvcopy\z@\global\setbox\@ne\lastbox}%
 \setbox\tw@\hbox{\unhbox\@ne\unskip\global\setbox\@ne\lastbox}%
 \setbox\tw@\hbox{%
 $\kern\wd\@ne\kern -\@tempdima\left\@firstoftwo#1%
 \if@borderstar\kern2pt\else\kern -\wd\@ne\fi%
 \global\setbox\@ne\vbox{\box\@ne\if@borderstar\else\kern 2\p@\fi}%
 \vcenter{\if@borderstar\else\kern -\ht\@ne\fi%
 \unvbox\z@\kern-\if@borderstar2\fi\baselineskip}%
 \if@borderstar\kern-2\@tempdima\kern2\p@\else\,\fi\right\@secondoftwo
#1 $%
 }\null \;\vbox{\kern\ht\@ne\box\tw@}%
 \endgroup
 }
 \makeatother

\title{On the Uniqueness of the Canonical Polyadic Decomposition of  third-order tensors --- Part I: Basic Results and Uniqueness of One Factor Matrix
\thanks{Research supported by: (1) Research Council KU Leuven:
GOA-Ambiorics,  GOA-MaNet,  CoE EF/05/006 Optimization in Engineering (OPTEC),
CIF1,  STRT 1/08/23,  (2) F.W.O.: (a) project G.0427.10N,  (b) Research Communities
ICCoS,  ANMMM and MLDM,  (3) the Belgian Federal Science Policy
Office: IUAP P6/04 (DYSCO,  ``Dynamical systems,  control and
optimization'',  2007--2011),  (4) EU: ERNSI.}}

\author{Ignat Domanov\footnotemark[2] \footnotemark[3]
\and
Lieven De Lathauwer\footnotemark[2] \footnotemark[3]}

\begin{document}

\maketitle

\renewcommand{\thefootnote}{\fnsymbol{footnote}}

\footnotetext[2]{Group Science,  Engineering and Technology,  KU Leuven--Kulak,
E. Sabbelaan 53,  8500 Kortrijk,  Belgium,
 ({\tt ignat.domanov,\ lieven.delathauwer@kuleuven-kulak.be}).}
\footnotetext[3]{Department of Electrical Engineering (ESAT),  SCD--SISTA,  KU Leuven,
Kasteelpark Arenberg 10,  postbus 2440,  B-3001 Heverlee (Leuven),  Belgium.}

\begin{abstract}
Canonical Polyadic Decomposition (CPD) of a higher-order tensor is  decomposition in a minimal number of rank-$1$ tensors. We give an overview of existing results concerning uniqueness. We present new, relaxed, conditions that guarantee uniqueness of one factor matrix. These conditions involve Khatri-Rao products of compound matrices. We make links with existing results involving ranks and k-ranks of factor matrices. We give a shorter proof, based on properties of second compound matrices, of existing results concerning overall CPD uniqueness in the case where one factor matrix has full column rank. We develop basic material involving $m$-th compound matrices that will be instrumental in Part II for establishing overall CPD uniqueness in cases where none of the factor matrices has full column rank.
\end{abstract}



\pagestyle{myheadings}
\thispagestyle{plain}
\markboth{IGNAT DOMANOV AND LIEVEN DE LATHAUWER}{On the Uniqueness of the Canonical Polyadic Decomposition --- Part I}

\section{Introduction}
\subsection{Problem statement}\label{Subsection1.1}
Throughout the paper $\mathbb F$ denotes the field of real or complex numbers;
$(\cdot)^T$ denotes transpose; $r_{\mathbf A}$ and $\textup{range}(\mathbf A)$ denote the rank and the range of a matrix $\mathbf A$, respectively;
$\textup{\text{Diag}}(\mathbf d)$ denotes a square diagonal matrix with the elements of a vector $\mathbf d$ on the main diagonal;
$\omega(\mathbf d)$ denotes the number of nonzero components of $\mathbf d$;
$C_n^k$ denotes the binomial coefficient,  $C_n^k=\frac{n!}{k!(n-k)!}$;
$\mzero_{m\times n}$, $\vzero_m$, and $\mathbf I_n$ are the zero $m\times n$ matrix, the zero $m\times 1$ vector, and the $n\times n$ identity matrix, respectively.

 We have the following basic definitions.
\begin{definition}\label{Def: outer product}
A third order-tensor $\mathcal T\in\mathbb F^{I\times J\times K}$ is {\em rank-$1$}
if it equals the outer product of three nonzero vectors $\mathbf a\in\mathbb F^I$,
$\mathbf b\in\mathbb F^J$ and $\mathbf c\in\mathbb F^K$, which means that $t_{ijk} = a_i b_j c_k$ for all values of the indices.
\end{definition}

A rank-1 tensor is also called a {\em simple tensor} or a {\em decomposable tensor}. The outer product in the definition is written as
$\mathcal T=\mathbf a\circ\mathbf b\circ\mathbf c$.

\begin{definition}
A {\em Polyadic Decomposition} (PD) of a third-order tensor $\mathcal T\in\mathbb F^{I\times J\times K}$ expresses $\mathcal T$ as a sum of  rank-$1$ terms:
\begin{equation}
\mathcal T=\sum\limits_{r=1}^R\mathbf a_r\circ \mathbf b_r\circ \mathbf c_r,
 \label{eqintro2}
\end{equation}
where $\mathbf a_r \in \mathbb F^{I}$, $\mathbf b_r \in \mathbb F^{J}$, $\mathbf c_r \in \mathbb F^{K}$, $1 \leq r \leq R$.
\end{definition}

We call the matrices
$\mathbf A =\left[\begin{matrix}\mathbf a_1&\dots&\mathbf a_R\end{matrix}\right] \in\mathbb F^{I\times R}$, $\mathbf B =\left[\begin{matrix}\mathbf b_1&\dots&\mathbf b_R\end{matrix}\right]\in\mathbb F^{J\times R}$ and $\mathbf C =\left[\begin{matrix}\mathbf c_1&\dots&\mathbf c_R\end{matrix}\right]\in\mathbb F^{K\times R}$
the  first, second and third factor matrix of $\mathcal T$, respectively. We also write (\ref{eqintro2}) as $\mathcal T=[\mathbf A,\mathbf B,\mathbf C]_R$.

\begin{definition}
The {\em rank} of a tensor $\mathcal T \in\mathbb F^{I\times J\times K}$ is defined
as the minimum number of rank-$1$ tensors in a PD  of  $\mathcal T$ and is denoted by $r_{\mathcal T}$.
\end{definition}

In general, the rank of a third-order tensor depends on $\mathbb F$ \cite{Kruskal1989}: a tensor over $\mathbb R$ may have a different rank than the
same tensor considered over $\mathbb C$.

\begin{definition}
A {\em Canonical Polyadic  Decomposition} (CPD) of a third-order tensor $\mathcal T$ expresses $\mathcal T$ as a minimal sum of rank-$1$ terms.
\end{definition}

Note that $\mathcal T=[\mathbf A,\mathbf B,\mathbf C]_R$ is a CPD of $\mathcal T$ if and only if $R=r_{\mathcal T}$.

Let us reshape $\mathcal T$ into  a vector $\mathbf t\in\mathbb F^{IJK\times 1}$ and a matrix $\mathbf T\in\mathbb F^{IJ\times K}$ as follows:
the $\left(i,j,k\right)$-th entry of $\mathcal T$ corresponds to
the $\left((i-1)JK+(j-1)K+k\right)$-th entry of $\mathbf t$ and to
the $\left((i-1)J+j,k\right)$-th entry of $\mathbf T$.
In particular, the rank-1 tensor $\mathbf a\circ\mathbf b\circ\mathbf c$ corresponds to
the vector $\mathbf a\otimes \mathbf b\otimes \mathbf c$ and to the rank-1 matrix $(\mathbf a\otimes\mathbf b)\mathbf c^T$,
where ``$\otimes$'' denotes the {\em Kronecker product}:
$$
\mathbf{a}  \otimes \mathbf{b} = \left[\begin{matrix}a_1\mathbf b^T&\dots& a_I\mathbf b^T\end{matrix}\right]^T=
\left[\begin{matrix}a_1b_1\dots a_1b_J&\dots& a_Ib_1\dots a_Ib_J\end{matrix}\right]^T.
$$
Thus, \eqref{eqintro2} can be identified either with
\begin{equation}
\mathbf t=\sum\limits_{r=1}^R\mathbf a_r\otimes \mathbf b_r\otimes \mathbf c_r,
\label{decomposition11}
\end{equation}
or with the matrix decomposition
\begin{equation}
\mathbf T=\sum\limits_{r=1}^R(\mathbf a_r\otimes \mathbf b_r){\mathbf c}_r^T.
\label{decomposition21}
\end{equation}
Further, \eqref{decomposition21} can be rewritten as a factorization of $\mathbf T$,
\begin{equation}
\mathbf T=(\mathbf A\odot\mathbf B)\mathbf C^T,
\label{eqT_V}
\end{equation}
where ``$\odot$'' denotes the {\em Khatri-Rao product} of matrices:
   $$
   \mathbf A\odot\mathbf B :=
   [\mathbf a_1\otimes\mathbf b_1\ \cdots\ \mathbf a_R\otimes\mathbf b_R]\in \mathbb F^{I J\times R}.
   $$
It is clear that in \eqref{eqintro2}--\eqref{decomposition21} the rank-1 terms can be arbitrarily permuted and that vectors within the same rank-1 term can be arbitrarily scaled provided the overall rank-1 term remains the same. {\em The CPD of a tensor is  unique} when it is only subject to these trivial indeterminacies.

In this paper we find sufficient conditions on the matrices $\mathbf A$, $\mathbf B$, and $\mathbf C$ which guarantee that
the CPD of $\mathcal T=[\mathbf A,\mathbf B,\mathbf C]_R$ is partially unique in the following sense:
the third factor matrix of any other CPD of $\mathcal T$ coincides with $\mathbf C$ up to permutation and  scaling of columns.
In such a case we  say that {\em the third factor matrix of $\mathcal T$ is unique}.
We also develop basic material involving $m$-th compound matrices that will be instrumental in Part II for establishing overall CPD uniqueness.

\subsection{Literature overview}\label{Subsection1.3}
The CPD was introduced by F.L. Hitchcock in \cite{Hitchcock}. It has been
rediscovered a number of times and called {\em Canonical Decomposition} (Candecomp) \cite{1970_Carroll_Chang}, {\em Parallel Factor Model} (Parafac) \cite{Harshman1970,1994HarshmanLundy}, and {\em Topographic Components Model}  \cite{1988Topographic}.
Key to many applications are the uniqueness properties of the CPD. Contrary to the matrix case, where there exist (infinitely) many rank-revealing decompositions, CPD may be unique without imposing constraints like orthogonality. Such constraints cannot always be justified from an application point of view. In this sense, CPD may be a meaningful data representation, and actually reveals a unique decomposition of the data in interpretable components. CPD has found many applications in Signal Processing \cite{Cichocki2009},\cite{ComoJ10}, Data Analysis \cite{Kroonenberg2008}, Chemometrics \cite{smilde2004multi}, Psychometrics \cite{1970_Carroll_Chang}, etc. We refer to the overview papers  \cite{KoldaReview,2009Comonetall,Lieven_ISPA} and the references therein for background,  applications and algorithms. We also refer to \cite{Sorber} for a discussion of optimization-based algorithms.
\subsubsection{Early results on uniqueness of the CPD}
In \cite[p. 61]{Harshman1970} the following result concerning the uniqueness of the CPD  is attributed to R. Jennrich.
\begin{theorem}
Let  $\mathcal T=[\mathbf A,\mathbf B,\mathbf C]_R$  and let
\begin{equation}
r_{\mathbf A}=r_{\mathbf B}=r_{\mathbf C}=R.\label{eqrarbrc=r}
\end{equation}
Then $r_{\mathcal T}=R$ and the CPD $\mathcal T=[\mathbf A,\mathbf B,\mathbf C]_R$ is  unique.
\end{theorem}

Condition \eqref{eqrarbrc=r} may be relaxed as follows.

\begin{theorem} \cite{Harshman1972}
Let  $\mathcal T=[\mathbf A,\mathbf B,\mathbf C]_R$, let
\begin{equation*}
r_{\mathbf A}=r_{\mathbf B}=R\ \ \text{and let any two columns of }\ \mathbf C\ \text{be linearly independent}.
\end{equation*}
Then $r_{\mathcal T}=R$ and the CPD $\mathcal T=[\mathbf A,\mathbf B,\mathbf C]_R$ is  unique.
\end{theorem}
\subsubsection{Kruskal's conditions}
A further relaxed result is due to J. Kruskal. To present Kruskal's theorem we recall the definition of $k$-rank (``$k$'' refers to ``Kruskal'').

\begin{definition}\label{defkrank}
The $k$-rank  of a matrix $\mathbf A$ is the largest number $k_{\mathbf A}$ such that
every subset of $k_{\mathbf A}$ columns of the matrix $\mathbf A$ is  linearly independent.
\end{definition}

Obviously, $k_{\mathbf A}\leq r_{\mathbf A}$.
Note that the notion of the $k$-rank is closely related to the  notions of
girth, spark, and $k$-stability \cite[Lemma 5.2, p. 317]{LimCommon2010} and  references therein.
The famous Kruskal theorem states the following.

\begin{theorem}\cite{Kruskal1977}\label{theoremKruskal}
Let $\mathcal T=[\mathbf A,\mathbf B,\mathbf C]_R$  and let
\begin{equation}
k_{\mathbf A}+k_{\mathbf B}+k_{\mathbf C}\geq 2R+2.
 \label{Kruskal}
 \end{equation}
Then $r_{\mathcal T}=R$ and the CPD of  $\mathcal T=[\mathbf A,\mathbf B,\mathbf C]_R$ is  unique.
 \end{theorem}

Kruskal's original proof was made more accessible in \cite{Stegeman2007} and was simplified in \cite[Theorem 12.5.3.1, p. 306]{Landsberg}. In \cite{Rhodes2010} an other proof of Theorem \ref{theoremKruskal} is given.

Before Kruskal arrived at Theorem \ref{theoremKruskal} he obtained results about uniqueness of one factor matrix
\cite[Theorem 3a--3d, p. 115--116]{Kruskal1977}.
These results were flawed. Here we present their corrected versions.

\begin{theorem}\label{theorem_uni-mode}\cite[Theorem 2.3]{GuoMironBrieStegeman} (for original formulation see \cite[Theorems 3a,b]{Kruskal1977})
Let  $\mathcal T=[\mathbf A,\mathbf B,\mathbf C]_R$ and suppose
\begin{equation}\label{unique_one_new_paper}
\begin{cases}
k_{\mathbf C}\geq 1,\\
r_{\mathbf C}+\min(k_{\mathbf A}, k_{\mathbf B})\geq R+2,\\
r_{\mathbf C}+k_{\mathbf A}+k_{\mathbf B}+\max(r_{\mathbf A}-k_{\mathbf A}, r_{\mathbf B}-k_{\mathbf B})\geq 2R+2.
\end{cases}
\end{equation}
Then $r_{\mathcal T}=R$ and  the third factor matrix of $\mathcal T$ is unique.
\end{theorem}

Let the matrices $\mathbf A$ and $\mathbf B$ have $R$ columns. Let  $\tilde{\mathbf A}$ be any set of columns of $\mathbf A$, let $\tilde{\mathbf B}$ be the corresponding set of columns of $\mathbf B$.
We will say that condition \cond{H}{m}
holds for the matrices $\mathbf A$ and $\mathbf B$ if
\begin{equation}
H(\delta):=\min\limits_{card(\tilde{\mathbf A})=\delta}\left[r_{\tilde{\mathbf A}}+r_{\tilde{\mathbf B}}-\delta\right]\geq\min(\delta,m)\quad\text{for}\quad
\delta=1,2,\dots,R.\tag{H{\scriptsize m}}
\end{equation}
\begin{theorem} (see \S \ref{Section4}, for original formulation see \cite[Theorems 3d]{Kruskal1977})\label{theoremKruskalnew3}
Let $\mathcal T=[\mathbf A, \mathbf B, \mathbf C]_R$ and $m:=R-r_{\mathbf C}+2$. Assume that
\begin{romannum}
\item $k_{\mathbf C}\geq 1$;
\item
\cond{H}{m} holds for $\mathbf A$ and $\mathbf B$.
\end{romannum}
Then $r_{\mathcal T}=R$ and the third factor matrix of  $\mathcal T$  is unique.
\end{theorem}

Kruskal also obtained  results about overall uniqueness that are
more general than Theorem \ref{theoremKruskal}. These results will be discussed in Part II \cite{partII}.

\subsubsection{Uniqueness of the CPD when one factor matrix has full column rank}
We say that a $K\times R$ matrix {\em has  full column rank } if its column rank is $R$, which implies $K \geq R$.

Let us assume that $r_{\mathbf C}=R$.
The following result concerning uniqueness of the CPD was obtained by  T. Jiang and N. Sidiropoulos in \cite{JiangSid2004}.
We reformulate  the result in terms  of the Khatri-Rao product of the second compound  matrices of $\mathbf A$ and $\mathbf B$. The $k$-th compound matrix of
an
$I\times R$ matrix $\mathbf A$ (denoted by $\mathcal C_k(\mathbf A)$) is the $C^k_I\times C^k_R$ matrix containing the determinants of all $k\times k$ submatrices of $\mathbf A$, arranged with the submatrix index sets in lexicographic order (see Definition \ref{defcompound} and Example \ref{Example2.2}).
\begin{theorem}\label{th:1.16}\cite[Condition A, p. 2628, Condition B and eqs. (16) and (17), p. 2630]{JiangSid2004}
Let $\mathbf A\in \mathbb F^{I\times R}$,  $\mathbf B\in \mathbb F^{J\times R}$,
$\mathbf C\in \mathbb F^{K\times R}$ and $r_{\mathbf C}=R$. Then the following statements are equivalent:
\begin{romannum}
\item if $\mathbf d\in\mathbb F^R$ is such that  $r_{\mathbf A \textup{\text{Diag}}(\mathbf d) \mathbf B^T}\leq 1$,  then $\omega(\mathbf d)\leq 1$;
\item if  $\mathbf d\in\mathbb F^R$ is such that
$$
(\mathcal C_2(\mathbf A)\odot \mathcal C_2(\mathbf B))\left[\begin{matrix}d_1d_2&d_1d_3&\dots&d_1d_R&d_2d_3&\dots&d_{R-1}d_R\end{matrix}\right]^T=
\vzero,
$$
  then $\omega(\mathbf d)\leq 1$;\hfill\cond{U}{2}
\item $r_{\mathcal T}=R$ and the CPD of $\mathcal T=[\mathbf A,  \mathbf B,  \mathbf C]_R$ is unique.
\end{romannum}
\end{theorem}

Papers  \cite{JiangSid2004} and  \cite{DeLathauwer2006}  contain the following more restrictive sufficient condition for CPD uniqueness, formulated differently. This condition can be expressed in terms of second compound matrices as follows.

\begin{theorem}\label{Theorem1.12}\cite[Remark 1, p. 652]{DeLathauwer2006}, \cite{JiangSid2004}
Let  $\mathcal T=[\mathbf A,\mathbf B,\mathbf C]_R$,  $r_{\mathbf C}=R$, and suppose
\begin{equation}
\mathbf U =\mathcal C_2(\mathbf A)\odot \mathcal C_2(\mathbf B) \text{       has full column rank}. \tag{C{\scriptsize 2}}
\end{equation}
Then $r_{\mathcal T}=R$ and the CPD  of $\mathcal T$ is  unique.
\end{theorem}

It is clear that \cond{C}{2} implies \cond{U}{2}.
If $r_{\mathbf C}=R$, then Kruskal's condition \eqref{Kruskal} is more restrictive than  condition  \cond{C}{2}.

\begin{theorem}\label{TheoremAlwin}\cite[Proposition 3.2, p. 215 and Lemma 4.4, p. 221]{Stegeman2009}
Let  $\mathcal T=[\mathbf A,\mathbf B,\mathbf C]_R$ and let  $r_{\mathbf C}=R$. If
\begin{equation}\tag{K{\scriptsize 2}}
\left\{
  \begin{array}{rl}
    r_{\mathbf A}+k_{\mathbf B}&\geq R+2,\\
    k_{\mathbf A}&\geq 2
  \end{array}
\right.
\qquad\text{ or }\qquad
\left\{
  \begin{array}{rl}
      r_{\mathbf B}+k_{\mathbf A}&\geq R+2,\\
      k_{\mathbf B}&\geq 2,
  \end{array}
\right.
\end{equation}
then \cond{C}{2} holds.
Hence, $r_{\mathcal T}=R$ and the CPD of $\mathcal T$ is  unique.
\end{theorem}

Theorem \ref{TheoremAlwin} is due to A. Stegeman
\cite[Proposition 3.2, p. 215 and Lemma 4.4, p. 221]{Stegeman2009}.
Recently, another proof of Theorem \ref{TheoremAlwin} has been obtained in
\cite[Theorem 1,  p. 3477]{Guo_Miron_Brie_Zhu_Liao2011}.

Assuming $r_{\mathbf C}=R$, the conditions of Theorems \ref{theoremKruskal} through \ref{TheoremAlwin} are related by
\begin{equation}\label{manyimplications2}
\begin{split}
k_{\mathbf A}+k_{\mathbf B}+k_{\mathbf C}\geq 2R+2\ \Rightarrow \cond{K}{2}\ &\Rightarrow \cond{C}{2}\
\Rightarrow
 \cond{U}{2}\ \\
 &\Leftrightarrow \textup{$r_{\mathcal T}=R$ and the CPD of  }\mathcal T \textup{ is unique.}
\end{split}
\end{equation}

\subsubsection{Necessary conditions for  uniqueness of the CPD. Results concerning  rank and $k$-rank of Khatri-Rao product}\label{subsubsectionnecessity}
It was shown in \cite{TenBerge2002} that  condition \eqref{Kruskal} is not only sufficient but also necessary for the uniqueness of the CPD if $R=2$ or $R=3$.
Moreover, it was proved in  \cite{TenBerge2002} that if
$R=4$ and if the $k$-ranks of the factor matrices coincide with their ranks, then the CPD of
$[\mathbf A,\mathbf B,\mathbf C]_4$ is  unique if and only if condition
\eqref{Kruskal} holds. Passing to higher values of $R$  we have the  following theorems.
\begin{theorem}\label{NecessaryCPD}\cite[p. 651]{Strassen 1983}, \cite[p. 2079, Theorem 2]{LiuSid2001},\cite[p. 28]{Krijnen1991}
Let  $\mathcal T=[\mathbf A,\mathbf B,\mathbf C]_R$, $r_{\mathcal T}=R\geq 2$, and let the CPD of $\mathcal T$ be unique.
Then
\begin{romannum}
\item
$\mathbf A\odot\mathbf B$, $\mathbf B\odot\mathbf C$, $\mathbf C\odot\mathbf A$  have full column rank;
\item
$\min(k_{\mathbf A},k_{\mathbf B},k_{\mathbf C})\geq 2$.
\end{romannum}
\end{theorem}
\begin{theorem}\label{necessityU2}\cite[Theorem 2.3]{LievenLL1}
Let  $\mathcal T=[\mathbf A,\mathbf B,\mathbf C]_R$, $r_{\mathcal T}=R\geq 2$, and let the CPD of $\mathcal T$ be unique.
Then the condition \textup{(U{\scriptsize 2})} holds for the  pairs $(\mathbf A,\mathbf B)$, $(\mathbf B,\mathbf C)$, and $(\mathbf C,\mathbf A)$.
\end{theorem}

Theorem \ref{necessityU2} gives more restrictive uniqueness conditions than Theorem  \ref{NecessaryCPD} and generalizes the
implication \textup{(iii)}$\Rightarrow$\textup(ii) of Theorem \ref{th:1.16} to CPDs with $r_{\mathbf C}\leq R$.

The following lemma gives a condition under which
\begin{align}
\mathbf A\odot\mathbf B \text{         has full column rank.}
\tag{C{\scriptsize 1}}
\end{align}
\begin{lemma}\label{lemmaMiron1}  \cite[Lemma 1, p. 3477]{Guo_Miron_Brie_Zhu_Liao2011}
Let $\mathbf A\in\mathbb F^{I\times R}$ and $\mathbf B\in\mathbb F^{J\times R}$.
If
\begin{align}
& \left\{
  \begin{array}{rl}
    r_{\mathbf A}+k_{\mathbf B}&\geq R+1,\\
    k_{\mathbf A}&\geq 1
  \end{array}
\right.
\qquad\text{ or }\qquad
\left\{
  \begin{array}{rl}
      r_{\mathbf B}+k_{\mathbf A}&\geq R+1,\\
      k_{\mathbf B}&\geq 1,
  \end{array}
\right.
\tag{K{\scriptsize 1}}
\end{align}
then \cond{C}{1} holds.
\end{lemma}

We conclude this section by mentioning two important corollaries that we will use.

\begin{corollary}\label{rank Khatri-Rao lemma}\cite[Lemma 1,  p. 2382]{Senior00parallelfactor}
If $k_{\mathbf A}+k_{\mathbf B}\geq R+1$,  then \cond{C}{1} holds.
\end{corollary}
\begin{corollary}\label{k-rank Khatri-Rao lemma}
\cite[Lemma 1,  p. 231]{890366}
 If  $k_{\mathbf A}\geq 1$ and  $k_{\mathbf B}\geq 1$,  then\\
$
k_{\mathbf A\odot \mathbf B}\geq \min (k_{\mathbf A}+k_{\mathbf B}-1, R).
$
\end{corollary}

The proof of Corollary \ref{k-rank Khatri-Rao lemma} in \cite{890366} was  based on Corollary \ref{rank Khatri-Rao lemma}.
Other proofs are given in \cite[Lemma 1,  p. 231]{CEM:CEM587} and
\cite[Lemma 3.3,  p. 544]{Stegeman2007}. (The proof in \cite{Stegeman2007}  is due to J. Ten Berge, see also \cite{TenBerge2000}.) All mentioned proofs are based on the Sylvester rank  inequality.

\subsection{Results and organization}\label{Subsection1.4}
Motivated by the conditions
 appearing in the various theorems of the preceding section, we formulate more
 general versions, depending on an integer parameter $m$. How these conditions,
 in conjunction with other assumptions, imply the uniqueness of  one particular factor matrix
 will be the core of our work.

 To introduce the new conditions we need the following notation. With a vector $\mathbf d=\left[\begin{matrix}d_{1}&\dots& d_R\end{matrix}\right]^T$
we associate the vector
\begin{equation}
\hatdSmR{m}{R}:=
 \left[\begin{matrix}d_{1}\cdots d_m&d_{1}\cdots d_{m-1}d_{m+1}&\dots& d_{R-m+1}\cdots d_R\end{matrix}\right]^T\in\mathbb F^{C^m_R},
  \label{eqd_big_product}
\end{equation}
whose entries are all products $d_{i_1}\cdots d_{i_m}$ with $1\leq i_1<\dots<i_m\leq R$.
 Let us define conditions \cond{K}{m}, \cond{C}{m}, \cond{U}{m} and \cond{W}{m}, which
  depend on matrices $\mathbf A\in\mathbb F^{I\times R}$, $\mathbf B\in\mathbb F^{J\times R}$, $\mathbf C\in\mathbb F^{K\times R}$ and an integer parameter $m$:
\begin{align}
&\left\{
  \begin{array}{rl}
    r_{\mathbf A}+k_{\mathbf B}&\geq R+m,\\
    k_{\mathbf A}&\geq m
  \end{array}
\right.
\qquad\text{ or }\qquad
\left\{
  \begin{array}{rl}
      r_{\mathbf B}+k_{\mathbf A}&\geq R+m,\\
      k_{\mathbf B}&\geq m
  \end{array}
\right.
;
\tag{K{\scriptsize m}}
\\
&\ \quad\mathcal C_{m}(\mathbf A)\odot \mathcal C_{m}(\mathbf B)\quad\ \ \text{ has full column rank}\tag{C{\scriptsize m}};
\\
&\begin{cases}
(\mathcal C_{m}(\mathbf A )\odot\mathcal C_m(\mathbf B))\hatdSmR{m}{R}=\vzero,\\
\mathbf d\in\mathbb F^R
\end{cases}\Rightarrow\quad
\hatdSmR{m}{R}=\vzero;
\tag{U{\scriptsize m}}
\\
& \begin{cases}
(\mathcal C_{m}(\mathbf A)\odot \mathcal C_m(\mathbf B))\hatdSmR{m}{R}=\vzero,\\
\mathbf d\in\textup{range}(\mathbf C^T)
\end{cases} \Rightarrow\quad
\hatdSmR{m}{R}=\vzero.\tag{W{\scriptsize m}}
\end{align}
In \S \ref{Preliminaries} we give a formal definition of compound matrices and present some of their properties. This basic material will be heavily used in the following sections.

In \S \ref{Section3} we establish the following implications:
\begin{equation}
\begin{matrix}
&  (\text{W{\scriptsize m}})       &\
            &(\text{W{\scriptsize m-1}})   &\          &\dots&\          &(\text{W{\scriptsize 2}}) &\  &(\text{W{\scriptsize 1}})\\
\text{(Lemma } \ref{PropositionA1})
&   \Uparrow  &\           &\Uparrow  &\          &\dots&\          &\Uparrow  &\           &\Uparrow\\
\text{(Lemma } \ref{compoundumuk})
&   (\text{U{\scriptsize m}})       &\Rightarrow           &(\text{U{\scriptsize m-1}})   &\Rightarrow          &\dots&\Rightarrow          &(\text{U{\scriptsize 2}})
     &\Rightarrow  &(\text{U{\scriptsize 1}})\\
\text{(Lemma } \ref{C_mU_m})
&    \Uparrow  &\           &\Uparrow  &\          &\dots&\          &\Uparrow  &\           &\Updownarrow\\
\text{(Lemma } \ref{compoundkhr})
&    (\text{C{\scriptsize m}})       &\Rightarrow &(\text{C{\scriptsize m-1}})   &\Rightarrow&\dots&\Rightarrow&(\text{C{\scriptsize 2}})       &\Rightarrow
      &(\text{C{\scriptsize 1}})\\
\text{(Lemma } \ref{compoundkhrkrusk})
&     \Uparrow  &\           &\Uparrow  &\          &\dots&\          &\Uparrow  &\           &\Uparrow\\
\text{(Lemma } \ref{prop:KmKk})
&   (\text{K{\scriptsize m}})       &\Rightarrow &(\text{K{\scriptsize m-1}})   &\Rightarrow&\dots&\Rightarrow&(\text{K{\scriptsize 2}})       &\Rightarrow
      &(\text{K{\scriptsize 1}})
\end{matrix}\label{maindiagramintro}
\end{equation}
as well as (Lemma \ref{compoundwmwk})
\begin{equation}
\text{if }\min(k_{\mathbf A}, k_{\mathbf B})\geq m-1 \text{, then}\
(\text{W{\scriptsize m}})       \Rightarrow
         (\text{W{\scriptsize m-1}})  \Rightarrow          \dots\ \Rightarrow           (\text{W{\scriptsize 2}}) \Rightarrow(\text{W{\scriptsize 1}}).
         \label{eq1.14}
\end{equation}
We also show in Lemmas \ref{prop:HmHk},
\ref{KmHm}--\ref{HmUm} that \eqref{maindiagramintro} remains valid after replacing conditions
$(\textup{C{\scriptsize m}})$,\dots,$(\textup{C{\scriptsize 1}})$ and equivalence
$(\textup{C{\scriptsize 1}})\Leftrightarrow (\textup{U{\scriptsize 1}})$ by
conditions $(\textup{H{\scriptsize m}})$,\dots,$(\textup{H{\scriptsize 1}})$ and implication
$(\textup{H{\scriptsize 1}})\Rightarrow (\textup{U{\scriptsize 1}})$, respectively.

Equivalence of $(\text{C{\scriptsize 1}})$ and $(\text{U{\scriptsize 1}})$ is trivial, since the two conditions are the same. The implications $(\textup{K{\scriptsize 2}}) \ \Rightarrow (\textup{C{\scriptsize 2}}) \ \Rightarrow (\textup{U{\scriptsize 2}})$ already appeared in \eqref{manyimplications2}. The implication $(\textup{K{\scriptsize 1}}) \ \Rightarrow (\textup{C{\scriptsize 1}})$ was given in Lemma \ref{lemmaMiron1}, and the implications
$(\textup{K{\scriptsize m}}) \Rightarrow (\textup{H{\scriptsize m}}) \Rightarrow (\textup{U{\scriptsize m}})$ are implicitly contained in \cite{Kruskal1977}.
From the definition of conditions $(\text{K{\scriptsize m}})$  and $(\text{H{\scriptsize m}})$
it follows that   $r_{\mathbf A}+r_{\mathbf B}\geq R+m$. On the other hand, condition $(\text{C{\scriptsize m}})$
may hold for $r_{\mathbf A}+r_{\mathbf B}< R+m$. We do not know examples where
$(\text{H{\scriptsize m}})$ holds, but $(\text{C{\scriptsize m}})$ does not. We suggest that $(\text{H{\scriptsize m}})$
always implies  $(\text{C{\scriptsize m}})$.

In \S \ref{Section4} we present a number of
 results establishing the uniqueness of one factor matrix under various
 hypotheses including at least one of the conditions $(\textup{K{\scriptsize m}})$, $(\textup{H{\scriptsize m}})$, $(\textup{C{\scriptsize m}})$, $(\textup{U{\scriptsize m}})$ and $\text{\textup{(W{\scriptsize m})}}$.
 The results of this section can be summarized as:\\ if $k_{\mathbf C}\geq 1$ and $m=m_{\mathbf C}:=R-r_{\mathbf C}+2$, then
\begin{equation}\label{manyimplicationsm}
\begin{aligned}
&
\begin{tikzcd}[column sep=tiny,row sep=tiny]
& & (\textup{C{\scriptsize m}}) \arrow[Rightarrow]{dr} & \\
\eqref{unique_one_new_paper}\ \Leftrightarrow & (\textup{K{\scriptsize m}}) \arrow[Rightarrow]{ur}\arrow[Rightarrow]{dr} & & (\textup{U{\scriptsize m}})\\
& & (\textup{H{\scriptsize m}}) \arrow[Rightarrow]{ur} &
\end{tikzcd}
\Rightarrow \
\begin{cases}
\mathbf A\odot\mathbf B \text{ has full column rank},\\
(\textup{W{\scriptsize m}}),\\
\min(k_{\mathbf A},k_{\mathbf B})\geq m-1
\end{cases}
\\
&\Rightarrow \
\begin{cases}
\mathbf A\odot\mathbf B \text{ has full column rank},\\
(\textup{W{\scriptsize m}}),(\textup{W{\scriptsize m-1}}),\dots,(\textup{W{\scriptsize 1}})
\end{cases}
 \\
&\Rightarrow\ \  r_{\mathcal T}=R \ \text{ and the third factor  matrix of}\ \mathcal T=[\mathbf A,\mathbf B,\mathbf C]_R\ \text{ is unique}.
\end{aligned}
\end{equation}
Thus, Theorems \ref{theorem_uni-mode}--\ref{theoremKruskalnew3} are
  implied by the more general statement \eqref{manyimplicationsm}, which therefore provides new,
   more relaxed sufficient conditions for uniqueness of one factor matrix.

Further, compare \eqref{manyimplicationsm} to \eqref{manyimplications2}.
For the case $r_{\mathbf C} = R$, i.e., $m =2$, uniqueness of the overall CPD has been established in Theorem \ref{th:1.16}.
Actually, in this case overall CPD uniqueness follows easily from uniqueness of ${\mathbf C}$.

In \S \ref{Section5} we simplify the proof of Theorem \ref{th:1.16} using the material we have developed so far.
In Part II \cite{partII} we will use \eqref{manyimplicationsm} to generalize \eqref{manyimplications2} to cases where possibly $r_{\mathbf C} < R$, i.e., $m > 2$.

\section{Compound matrices and their properties} \label{Preliminaries}
In this section we define compound matrices and present several of their properties. The material will be heavily used in the following sections.

Let
\begin{equation}
 S_n^k:=\{(i_1, \dots, i_{k}): 1\leq i_1<\dots <i_k\leq n\}\label{eqS_n^k}
\end{equation}
  denote the set of all $k$ combinations of the set $\{1, \dots, n\}$.
   We assume that the elements of
    $S_n^k$  are ordered lexicographically.
     Since the elements of $S_n^k$ can be indexed from $1$ up to $C^{k}_n$, there exists an order preserving bijection
   \begin{equation}\label{sigmabijection}
   \sigma_{n,k}:\{1,2, \dots, C^{k}_n\}\rightarrow S_{n}^k=\{S_n^k(1), S_n^k(2), \dots, S_n^k(C^{k}_n)\}.
   \end{equation}
   In the sequel we will both use indices taking values in $\{1,2, \dots, C^{k}_n\}$ and multi-indices taking values in $S_{n}^k$. The connection between both is given by \eqref{sigmabijection}.

To distinguish between vectors from $\mathbb F^R$  and $\mathbb F^{C_n^k}$ we will use  the subscript $S_{n}^k$, which will
 also indicate  that the  vector entries are enumerated by means of $S_{n}^k$.
For instance, throughout the paper the vectors $\mathbf d\in \mathbb F^R$ and ${\mathbf d}_{S_R^m}\in \mathbb F^{C_R^m}$ are always defined by
\begin{align}
\mathbf d=&\left[\begin{matrix}d_1&d_2&\dots&d_R\end{matrix}\right]\in\mathbb F^R,\nonumber\\
{\mathbf d}_{S_R^m}=&\left[\begin{matrix}d_{(1,\dots, m)}&\dots& d_{(j_1,\dots, j_m)}&\dots& d_{(R-m+1,\dots, R)}\end{matrix}\right]^T\in\mathbb F^{C^m_R}.
 \label{eqd_big}
\end{align}
Note that if  $d_{(i_1,\dots, i_m)}=d_{i_1}\cdots d_{i_m}$ for all indices $i_1,\dots,i_m$, then  the vector
${\mathbf d}_{S_R^m}$ is equal to the vector $\hatdSmR{m}{R}$ defined in \eqref{eqd_big_product}.

Thus,  ${\mathbf d}_{S_R^1}=\hatdSmR{1}{R}=\mathbf d$.

\begin{definition}\label{defcompound} \cite{HornJohnson}
Let $\mathbf A\in \mathbb F^{m\times n}$ and $k\leq\min(m,n)$. Denote by\\
$\mathbf A (S_m^k(i), S_m^k(j))$ the submatrix
at the intersection of the $k$ rows with row numbers $S_m^k(i)$ and the $k$ columns with column numbers $S_m^k(j)$.
The $C_m^k$-by-$C_n^k$ matrix whose $(i, j)$ entry is $\det\mathbf A (S_m^k(i), S_n^k(j))$ is
called the $k$-th compound matrix of $\mathbf A$ and is denoted by $\mathcal C_k(\mathbf A)$.
\end{definition}

\begin{example}\label{Example2.2}
Let
$$
\mathbf A=
\left[
\begin{matrix}
a_1&1&0&0\\
a_2&0&1&0\\
a_3&0&0&1
\end{matrix}
\right].
$$
Then
\begin{align*}
\mathcal C_2(\mathbf A) =& \left[
\begin{matrix}
\mathcal C_2(\mathbf A)_{1}&\mathcal C_2(\mathbf A)_{2}&\mathcal C_2(\mathbf A)_{3}&
\mathcal C_2(\mathbf A)_{4}&\mathcal C_2(\mathbf A)_{5}&\mathcal C_2(\mathbf A)_{6}
\end{matrix}
\right]
\\
=& \left[
\begin{matrix}
\mathcal C_2(\mathbf A)_{(1,2)}&\mathcal C_2(\mathbf A)_{(1,3)}&\mathcal C_2(\mathbf A)_{(1,4)}&
\mathcal C_2(\mathbf A)_{(2,3)}&\mathcal C_2(\mathbf A)_{(2,4)}&\mathcal C_2(\mathbf A)_{(3,4)}
\end{matrix}
\right]\\
=&\bordermatrix[{[]}]{%
&(1,2)&(1,3)&(1,4)&(2,3)&(2,4)&(3,4)\cr
(1,2)&
\Big|
\begin{matrix}
    a_1& 1\\
    a_2& 0
    \end{matrix}
\Big|
&
\Big|
\begin{matrix}
    a_1& 0\\
    a_2& 1
    \end{matrix}
\Big|
&
\Big|
\begin{matrix}
    a_1& 0\\
    a_2& 0
    \end{matrix}
\Big|
&
\Big|
\begin{matrix}
    1& 0\\
    0& 1
    \end{matrix}
\Big|
&
\Big|
\begin{matrix}
    1& 0\\
    0& 0
    \end{matrix}
\Big|
&
\Big|
\begin{matrix}
    0& 0\\
    1& 0
    \end{matrix}
\Big|
\cr
(1,3)&
\Big|
\begin{matrix}
    a_1& 1\\
    a_3& 0
    \end{matrix}
\Big|
&
\Big|
\begin{matrix}
    a_1& 0\\
    a_3& 0
    \end{matrix}
\Big|
&
\Big|
\begin{matrix}
    a_1& 0\\
    a_3& 1
    \end{matrix}
\Big|
&
\Big|
\begin{matrix}
    1& 0\\
    0& 0
    \end{matrix}
\Big|
&
\Big|
\begin{matrix}
    1& 0\\
    0& 1
    \end{matrix}
\Big|
&
\Big|
\begin{matrix}
    0& 0\\
    0& 1
    \end{matrix}
\Big|
\cr
(2,3)&
\Big|
\begin{matrix}
    a_2& 0\\
    a_3& 0
    \end{matrix}
\Big|
&
\Big|
\begin{matrix}
    a_2& 1\\
    a_3& 0
    \end{matrix}
\Big|
&
\Big|
\begin{matrix}
    a_2& 0\\
    a_3& 1
    \end{matrix}
\Big|
&
\Big|
\begin{matrix}
    0& 1\\
    0& 0
    \end{matrix}
\Big|
&
\Big|
\begin{matrix}
    0& 0\\
    0& 1
    \end{matrix}
\Big|
&
\Big|
\begin{matrix}
    1& 0\\
    0& 0
    \end{matrix}
\Big|
}\\
=& \left[
\begin{array}{rrrrrr}
-a_2& a_1&0&1&0&0\\
-a_3&0&a_1&0&1&0\\
0&-a_3&a_2&0&0&1
\end{array}
\right].
\end{align*}
\end{example}
Definition \ref{defcompound} immediately implies the following lemma.
\begin{lemma} \label{compoundprop1}
Let $\mathbf A\in\mathbb F^{I\times R}$ and $k\leq \min(I,R)$. Then
\begin{remunerate}
\item
$\mathcal C_1(\mathbf A)=\mathbf A$;
\item
If $I=R$, then $\mathcal C_R(\mathbf A)=\det(\mathbf A)$;
\item
 $\mathcal C_k(\mathbf A)$ has one or more zero columns if and only if
$k > k_{\mathbf A}$;
\item
$\mathcal C_k(\mathbf A)$ is equal to the zero matrix if and only if
$k > r_{\mathbf A}$.
\end{remunerate}
\end{lemma}

The following properties of compound matrices are well-known.
\begin{lemma}\label{LemmaCompound}\cite[p. 19--22]{HornJohnson}
Let $k$  be a positive integer and let $\mathbf A$ and $\mathbf B$ be
matrices such that $\mathbf A\mathbf B$, $\mathcal C_k(\mathbf A)$, and $\mathcal C_k(\mathbf B)$ are  defined.
Then
\begin{remunerate}
\item
$\mathcal C_k(\mathbf A\mathbf B) =\mathcal  C_k(\mathbf A)\mathcal C_k(\mathbf B)$ (Binet-Cauchy formula);
\item
If $\mathbf A$ is nonsingular square matrix,  then $\mathcal C_k(\mathbf A)^{-1} =\mathcal C_k(\mathbf A^{-1})$;
\item
$\mathcal C_k(\mathbf A^T) = (\mathcal C_k(\mathbf A))^T$;
\item
$\mathcal C_k(\mathbf I_n ) = \mathbf I_{C_n^k}$;
\item
If $\mathbf A$ is an $n\times n$ matrix, then
$\det (\mathcal C_k(\mathbf A)) = \det (\mathbf A)^{C_{n-1}^{k-1}}$ (Sylvester-Franke theorem).
\end{remunerate}
\end{lemma}
We will extensively use compound matrices of diagonal matrices.
\begin{lemma}\label{Lemma2.3}
Let  $\mathbf d\in\mathbb F^R$,  $k\leq R$, and let $\hatdSmR{k}{R}$
be defined by \eqref{eqd_big_product}.
Then
\begin{remunerate}
\item
$\hatdSmR{k}{R}=\vzero$ if and only if $\omega(\mathbf d)\leq k-1$;
\item
$\hatdSmR{k}{R}$ has exactly one nonzero component if and only if $\omega(\mathbf d)=k$;
\item
$\mathcal C_k(\textup{\text{Diag}}(\mathbf d))=\textup{\text{Diag}}(\hatdSmR{k}{R})$.
\end{remunerate}
\end{lemma}

\begin{example}
Let $\mathbf d=\left[\begin{matrix}d_1&d_2&d_3&d_4\end{matrix}\right]^T$ and $\mathbf D=\textup{\text{Diag}}(\mathbf d)$. Then
\begin{align*}
\mathcal C_2(\mathbf D)=&\textup{\text{Diag}}(\left[\begin{matrix}d_1d_2& d_1d_3& d_1d_4& d_2d_3& d_2d_4& d_3d_4\end{matrix}\right])=\textup{\text{Diag}}(\hatdSmR{2}{4}),\\
\mathcal C_3(\mathbf D)=&\textup{\text{Diag}}(\left[\begin{matrix}d_1d_2d_3& d_1d_2d_4& d_1d_3d_4& d_2d_3d_4\end{matrix}\right])=\textup{\text{Diag}}(\hatdSmR{3}{4}).
\end{align*}
\end{example}

For vectorization of a  matrix $\mathbf T=[\mathbf t_1\ \cdots\ \mathbf t_R]$, we follow the convention that $\textup{vec}(\mathbf T)$ denotes
the column vector obtained by stacking the columns of $\mathbf T$ on top of one another, i.e.,
$$
    \textup{vec}(\mathbf T) =
    \left[
    \begin{matrix}
    \mathbf t_1^T& \dots& \mathbf t_R^T
    \end{matrix}
    \right]^T.
$$
It is clear that in vectorized form, rank-1 matrices correspond to Kronecker products of two vectors. Namely, for arbitrary vectors $\mathbf a$ and $\mathbf b$, $\textup{\text{vec}}(\mathbf b\mathbf a^T)=\mathbf a\otimes\mathbf b$. For matrices $\mathbf A$ and $\mathbf B$ that both have $R$ columns and $\mathbf d\in \mathbb F^R$, we now immediately obtain
 expressions that we will frequently use:
\begin{gather}
\textup{\text{vec}}(\mathbf B \textup{\text{Diag}}(\mathbf d) \mathbf A^T)=
\textup{\text{vec}}\left(\sum\limits_{r=1}^R\mathbf b_r  \mathbf a_r^Td_r\right)=
\sum\limits_{r=1}^R(\mathbf a_r\otimes \mathbf b_r)d_r=
(\mathbf A\odot \mathbf B)\mathbf d,\label{eqmatrtovec}\\
\mathbf A \textup{\text{Diag}}(\mathbf d) \mathbf B^T=\mzero\ \Leftrightarrow \
\mathbf B \textup{\text{Diag}}(\mathbf d) \mathbf A^T=\mzero\ \Leftrightarrow \
(\mathbf A\odot \mathbf B)\mathbf d=\vzero.\label{eqmatrtovec_b}
\end{gather}
The following generalization of property \eqref{eqmatrtovec} will be used throughout the paper.
\begin{lemma}\label{cor2.6}
Let $\mathbf A\in \mathbb F^{I\times R}$,  $\mathbf B\in \mathbb F^{J\times R}$,  $\mathbf d\in\mathbb F^R$, and $k\leq\min (I,J,R)$. Then
$$
\textup{vec}(\mathcal C_k(\mathbf B \textup{\text{Diag}}(\mathbf d) \mathbf A^T))=
[\mathcal C_k(\mathbf A)\odot \mathcal C_k(\mathbf B)]\hatdSmR{k}{R},
$$
where $\hatdSmR{k}{R}\in\mathbb F^{C^k_R}$ is defined by \eqref{eqd_big_product}.
\end{lemma}

{\em Proof.}
From Lemma \ref{LemmaCompound} (1),(3) and Lemma \ref{Lemma2.3} (3)  it follows that
$$
\mathcal C_k(\mathbf B \textup{\text{Diag}}(\mathbf d) \mathbf A^T)=
\mathcal C_k(\mathbf B)\mathcal C_k( \textup{\text{Diag}}(\mathbf d))\mathcal C_k( \mathbf A^T)=
\mathcal C_k(\mathbf B) \textup{\text{Diag}}(\hatdSmR{k}{R})\mathcal C_k( \mathbf A)^T.
$$
By  \eqref{eqmatrtovec},
$$
\textup{vec}(\mathcal C_k(\mathbf B) \textup{\text{Diag}}(\hatdSmR{k}{R})\mathcal C_k( \mathbf A)^T)=
[\mathcal C_k(\mathbf A)\odot \mathcal C_k(\mathbf B)]\hatdSmR{k}{R}.\qquad\endproof
$$
The following Lemma contains an equivalent definition of condition \cond{U}{m}.
\begin{lemma}\label{lemma:equivUm}
Let $\mathbf A\in \mathbb F^{I\times R}$ and  $\mathbf B\in \mathbb F^{J\times R}$. Then the following statements are equivalent:
\begin{romannum}
\item if $\mathbf d\in\mathbb F^R$ is such that  $r_{\mathbf A \textup{\text{Diag}}(\mathbf d) \mathbf B^T}\leq m-1$,  then $\omega(\mathbf d)\leq m-1$;
\item \cond{U}{m} holds.
\end{romannum}
\end{lemma}
\begin{proof}
From the definition of the $m$-th compound matrix and Lemma \ref{cor2.6} it follows that
\begin{equation*}
\begin{split}
r_{\mathbf A \textup{\text{Diag}}(\mathbf d) \mathbf B^T}&=r_{\mathbf B \textup{\text{Diag}}(\mathbf d) \mathbf A^T}\leq m-1
\ \ \Leftrightarrow\ \
\mathcal C_m(\mathbf B \textup{\text{Diag}}(\mathbf d) \mathbf A^T)=\mzero\\
\ \ &\Leftrightarrow\ \
\text{vec}({\mathcal C_m(\mathbf B \textup{\text{Diag}}(\mathbf d) \mathbf A^T)})=\vzero
\ \ \Leftrightarrow\ \
[\mathcal C_m(\mathbf A)\odot \mathcal C_m(\mathbf B)]\hatdSmR{m}{R}=\vzero.
\end{split}
\end{equation*}
Now the result follows from Lemma \ref{Lemma2.3} (1).
\end{proof}

The following three auxiliary  lemmas will be used in \S \ref{Section3}.
\begin{lemma}\label{lemmaU_mKhR}
Consider $\mathbf A\in \mathbb F^{I\times R}$ and  $\mathbf B\in \mathbb F^{J\times R}$ and let condition
$\text{\textup{(U{\scriptsize m})}}$ hold. Then $\min (k_{\mathbf A}, k_{\mathbf B})\geq m$.
\end{lemma}
\begin{proof}
We prove equivalently that if $\min (k_{\mathbf A}, k_{\mathbf B})\geq m$ does not hold, then $\text{\textup{(U{\scriptsize m})}}$ does not hold. Hence, we start by assuming that $\min (k_{\mathbf A}, k_{\mathbf B})=k<m$, which implies that there exist indices $i_1, \dots, i_m$ such that the vectors $\mathbf a_{i_1}, \dots, \mathbf a_{i_m}$  or the vectors $\mathbf b_{i_1}, \dots, \mathbf b_{i_m}$ are linearly dependent.
Let
$$
\mathbf d:=\left[\begin{matrix}d_1&\dots&d_R\end{matrix}\right]^T, \qquad
d_i:=
\begin{cases}
1, &i\in\{i_1, \dots, i_m\};\\
0, &i\not\in\{i_1, \dots, i_m\},
\end{cases}
$$
 and let $\hatdSmR{m}{R}\in\mathbb F^{C^m_R}$ be given by
\eqref{eqd_big_product}.
Because of the way $\mathbf d$ is defined, $\hatdSmR{m}{R}$ has exactly one nonzero entry, namely $ d_{i_1}\cdots d_{i_m}$. We now have
$$
(\mathcal C_{m}(\mathbf A)\odot \mathcal C_m(\mathbf B))\hatdSmR{m}{R}=
\mathcal C_{m}(\left[\begin{matrix}\mathbf a_{i_1}&\dots&\mathbf a_{i_m}\end{matrix}\right])\odot
 \mathcal C_{m}(\left[\begin{matrix}\mathbf b_{i_1}&\dots&\mathbf b_{i_m}\end{matrix}\right])d_{i_1}\cdots d_{i_m}=\vzero,
$$
in which the latter equality holds because of the assumed linear dependence of $\mathbf a_{i_1}$, $\dots$, $\mathbf a_{i_m}$  or $\mathbf b_{i_1}$, $\dots$, $\mathbf b_{i_m}$. We conclude that condition $\text{(U{\scriptsize m})}$ does not hold.
\qquad\end{proof}
\begin{lemma}\label{compound1}
Let $m\leq I$. Then there exists a linear mapping $\Phi^{I,m}: \mathbb F^I\rightarrow \mathbb F^{C_I^{m}\times C_I^{m-1}}$ such that
\begin{equation}\label{eq3.4}
\mathcal C_{m}([\mathbf A\ \mathbf x])=\Phi^{I,m}(\mathbf x)\mathcal C_{m-1}(\mathbf A) \quad \text{ for all }\ \mathbf A\in \mathbb F^{I\times (m-1)}\
 \text{ and  for all }\ \mathbf x\in \mathbb F^I.
\end{equation}
\end{lemma}
{\em Proof.}
Since $[\mathbf A\ \mathbf x]$ has $m$ columns, $\mathcal C_{m}([\mathbf A\ \mathbf x])$ is a vector that contains the determinants of the matrices formed by $m$ rows. Each of these determinants can be expanded  along its last column, yielding linear combinations of $(m-1) \times (m-1)$ minors, the combination coefficients being equal to entries of $\mathbf x$, possibly up to the sign. Overall, the expansion can be written in the form (\ref{eq3.4}), in which $\Phi^{I,m}(\mathbf x)$ is a $C_I^{m}\times C_I^{m-1}$ matrix, the nonzero entries of which are equal to entries of $\mathbf x$, possibly up to the sign.

More in detail, we have the following.
\begin{romannum}
\item
Let $\widehat{\mathbf A}\in \mathbb F^{m\times (m-1)}$,  $\widehat{\mathbf x}\in \mathbb F^{m}$.
By the Laplace expansion theorem \cite[p. 7]{HornJohnson},
$$
\mathcal C_{m}([\widehat{\mathbf A}\ \widehat{\mathbf x}]) = \det([\widehat{\mathbf A}\ \widehat{\mathbf x}])=
\left[
\begin{matrix}
\widehat{x}_{m}& -\widehat{x}_{m-1}& \widehat{x}_{m-2}&\dots &(-1)^{m-1}\widehat{x}_1
\end{matrix}
\right]\mathcal C_{m-1}(\widehat{\mathbf A}).
$$
Hence,  Lemma \ref{compound1} holds  for $m=I$  with
$$
\Phi^{m,m}(\mathbf x) = \left[\begin{matrix}x_{m}& -x_{m-1}& x_{m-2}&\dots &(-1)^{m-1}x_1\end{matrix}\right].
$$
\item Let $m<I$.
Since, $\mathcal C_{m}([\mathbf A\ \mathbf x])=\left[\begin{matrix}d_1&\dots&d_{C^{m}_I}\end{matrix}\right]^T$,  it follows from
the definition of compound matrix that   $d_i=\mathcal C_{m}([\widehat{\mathbf A}\ \widehat{\mathbf x}])$,  where
$ [\widehat{\mathbf A}\ \widehat{\mathbf x}]$ is a submatrix of $[\mathbf A\ \mathbf x]$ formed by rows with the numbers $\sigma_{I,m}(i)=S_{I}^m(i):=(i_1, \dots, i_m)$.
Let us define $\Phi_i(\mathbf x)\in \mathbb F^{1\times C^{m-1}_I}$ by
\begin{equation*}
\begin{matrix}
\Phi_i(\mathbf x)=[&0        &\dots &0      &x_{i_{m}}                          & 0     &\dots  & 0     &(-1)^{m-1}x_{i_1}                &\dots ],  \\
                   &\uparrow &\dots &\dots  &\uparrow                           &\dots  &\dots  &\dots  &\uparrow                         &\dots \ \ \\
                   &1        &\dots &\dots  &j_m &\dots  &\dots  &\dots  &j_1 &\dots \ \
\end{matrix}
\end{equation*}
where
$$
j_m:=\sigma_{I,m-1}^{-1}((i_1, \dots, i_{m-1})),\quad\dots\quad,j_1:=\sigma_{I,m-1}^{-1}((i_2, \dots, i_{m}))
$$
and $\sigma_{I,m-1}^{-1}$ is defined by \eqref{sigmabijection}.
Then by (i),
 $$
 d_i=
\mathcal C_{m}([\widehat{\mathbf A}\ \widehat{\mathbf x}]) =
\big[{x}_{i_{m}}\ \ -{x}_{i_{m-1}}\ \ {x}_{i_{m-2}}\ \ \dots\ \ (-1)^{m-1}{x}_{i_1}\big]\mathcal C_{m-1}(\widehat{\mathbf A})=
\Phi_i(\mathbf x)\mathcal C_{m-1}(\mathbf A).
$$
The proof is completed by setting
$$
\Phi^{I,m}(\mathbf x)=
\left[\begin{matrix}
\Phi_1(\mathbf x)\\
\vdots\\
\Phi_{C_{I}^m}(\mathbf x)
\end{matrix}\right].\qquad\endproof
$$
\end{romannum}
\begin{example}
Let us illustrate Lemma \ref{compound1} for $m=2$ and $I=4$. If\\
$\mathbf A=\left[\begin{matrix}a_{11}& a_{21}& a_{31}& a_{41}\end{matrix}\right]^T$, then
\begin{align*}
\mathcal C_2(\big[\mathbf A\ \ \mathbf x\big])=
&\mathcal C_2(
\left[\begin{matrix}
a_{11}&x_1\\
a_{21}&x_2\\
a_{31}&x_3\\
a_{41}&x_4
\end{matrix}\right]
)=
\left[\begin{matrix}
x_2a_{11}-x_1a_{21}\\
x_3a_{11}-x_1a_{31}\\
x_4a_{11}-x_1a_{41}\\
x_3a_{21}-x_2a_{31}\\
x_4a_{21}-x_2a_{41}\\
x_4a_{31}-x_3a_{41}
\end{matrix}\right]=
\left[\begin{matrix}
x_2  &-x_1   &0     &0\\
x_3  &0      &-x_1  &0\\
x_4  &0      &0     &-x_1\\
0    &x_3    &-x_2   &0\\
0    &x_4    &0     &-x_2\\
0    &0      &x_4   &-x_3
\end{matrix}\right]
\left[\begin{matrix}
a_{11}\\
a_{21}\\
a_{31}\\
a_{41}
\end{matrix}\right]
\\
=&\Phi^{4,2}(\mathbf x)\mathcal C_1(\mathbf A).
\end{align*}
\end{example}

\section{Basic implications}\label{Section3}

In this section we derive the implications in \eqref{maindiagramintro} and \eqref{eq1.14}. We first establish scheme \eqref{maindiagramintro} by means of Lemmas \ref{C_mU_m}, \ref{prop:C1U1}, \ref{PropositionA1}, \ref{prop:KmKk}, \ref{compoundkhr}, \ref{compoundumuk} and \ref{compoundkhrkrusk}.

\begin{lemma} \label{C_mU_m}
Let $\mathbf A\in \mathbb F^{I\times R}$,  $\mathbf B\in \mathbb F^{J\times R}$,  and $2 \leq m \leq \min(I, J)$.
Then condition \textup{(C{\scriptsize m})} implies condition \textup{(U{\scriptsize m})}.
\end{lemma}
\begin{proof}
Since, by \textup{(C{\scriptsize m})}, $C_m(\mathbf A)\odot \mathcal C_m(\mathbf B)$ has only the zero vector in its kernel, it does a forteriori not have an other vector in its kernel with the structure specified in \textup{(U{\scriptsize m})}.
\end{proof}

\begin{lemma} \label{prop:C1U1}
For $\mathbf A\in \mathbb F^{I\times R}$ and $\mathbf B\in \mathbb F^{J\times R}$. Then
$$
\textup{(C{\scriptsize 1})}\Leftrightarrow \textup{(U{\scriptsize 1})}\Leftrightarrow \mathbf A\odot \mathbf B \text{      has full column rank.}
$$
\end{lemma}
\begin{proof}
The proof follows trivially from Lemma \ref{compoundprop1}.1, since $\hatdSmR{1}{R}=\mathbf d$.
\end{proof}
\begin{lemma}\label{PropositionA1}
Let $\mathbf A\in \mathbb F^{I\times R}$,  $\mathbf B\in \mathbb F^{J\times R}$,
and $1\leq m\leq \min(I, J)$.
Then condition $\text{\textup{(U{\scriptsize m})}}$ implies condition \textup{(W{\scriptsize m})}
for any matrix $\mathbf C\in \mathbb F^{K\times R}$.
\end{lemma}
\begin{proof}
The proof trivially follows from the definitions of conditions $\text{\textup{(U{\scriptsize m})}}$ and $\text{\textup{(W{\scriptsize m})}}$.
\end{proof}
\begin{lemma} \label{prop:KmKk}
Let $\mathbf A\in \mathbb F^{I\times R}$,  $\mathbf B\in \mathbb F^{J\times R}$,  and $1 < m\leq \min(I, J)$.
Then condition $\text{\textup{(K{\scriptsize m})}}$ implies conditions \textup{(K{\scriptsize m-1})},$\dots$,\textup{(K{\scriptsize 1})}.
\end{lemma}
\begin{proof}
Trivial.
\end{proof}
\begin{lemma} \label{prop:HmHk}
Let $\mathbf A\in \mathbb F^{I\times R}$,  $\mathbf B\in \mathbb F^{J\times R}$,  and $1 < m\leq \min(I, J)$.
Then condition $\text{\textup{({H\scriptsize m})}}$ implies conditions \textup{(H{\scriptsize m-1})},$\dots$,\textup{(H{\scriptsize 1})}.
\end{lemma}
\begin{proof}
Trivial.
\end{proof}

\begin{lemma}\label{compoundkhr}
Let $\mathbf A\in \mathbb F^{I\times R}$,  $\mathbf B\in \mathbb F^{J\times R}$,  and $1<m\leq \min(I, J)$.
Then condition $\text{\textup{(C{\scriptsize m})}}$ implies conditions \textup{(C{\scriptsize m-1})},$\dots$,\textup{(C{\scriptsize 1})}.
\end{lemma}
\begin{proof}
It is sufficient to prove that $\text{\textup{(C{\scriptsize k})}}$ implies $\text{\textup{(C{\scriptsize k-1})}}$ for
$k\in\{m,m-1,\dots,2\}$.
Let us assume that there exists  a vector ${\mathbf d}_{S_R^{k-1}}\in\mathbb F^{C^{k-1}_R}$ such that
$$
\left[\mathcal C_{k-1}(\mathbf A)\odot \mathcal C_{k-1}(\mathbf B)\right]{\mathbf d}_{S_R^{k-1}}=\vzero,
$$
which, by \eqref{eqmatrtovec_b}, is equivalent with
$$
\mathcal C_{k-1}(\mathbf A)\textup{\text{Diag}}({\mathbf d}_{S_R^{k-1}}) \mathcal C_{k-1}(\mathbf B)^T=\mzero.
$$
Multiplying by matrices
$\Phi^{I,k}(\mathbf a_r)\in\mathbb F^{C_I^{k}\times C_I^{k-1}}$
and $\Phi^{J,k}(\mathbf b_r)\in\mathbb F^{C_J^{k}\times C_J^{k-1}}$,
constructed as in Lemma \ref{compound1}, we obtain
$$
\Phi^{I,k}(\mathbf a_r)\mathcal C_{k-1}(\mathbf A) \textup{\text{Diag}}({\mathbf d}_{S_R^{k-1}}) \mathcal C_{k-1}(\mathbf B)^T\Phi^{J,k}(\mathbf b_r)^T=\mzero, \quad \quad r=1,\dots,R,
$$
which, by \eqref{eqmatrtovec_b}, is equivalent with
\begin{equation}
\label{eq6}
\left[\left(\Phi^{I,k}(\mathbf a_r)\mathcal C_{k-1}(\mathbf A)\right) \odot \left( \Phi^{J,k}(\mathbf b_r)\mathcal C_{k-1}(\mathbf B)\right)\right]{\mathbf d}_{S_R^{k-1}}=\vzero, \qquad r=1, \dots, R.
\end{equation}
By \eqref{eq3.4},
\begin{equation}
\begin{split}
\Phi^{I,k}(\mathbf a_r)&\mathcal C_{k-1}([\mathbf a_{i_1}\ \dots\ \mathbf a_{i_{k-1}}])=
\mathcal C_k([\mathbf a_{i_1}\ \dots\ \mathbf a_{i_{k-1}}\ \mathbf a_r])=\\
&\begin{cases}
\vzero, &\text{if } r\in\{i_1, \dots, i_{k-1}\};\\
\pm \mathcal C_l(\mathbf A)_{[i_1, i_2, \dots, i_{k-1}, r]}, &\text{if }r\not \in\{i_1, \dots, i_{k-1}\},
\end{cases}
\end{split}
\end{equation}
where
$\mathcal C_k(\mathbf A)_{[i_1, i_2, \dots, i_{k-1}, r]}$ denotes the
$[i_1, i_2, \dots, i_{k-1}, r]$-th column of the matrix\\ $\mathcal C_k(\mathbf A)$,
in which  $[i_1, i_2, \dots, i_{k-1}, r]$ denotes an ordered $k$-tuple.
(Recall that by \eqref{sigmabijection}, the columns of $\mathcal C_k(\mathbf A)$ can be enumerated with $S^k_R$.)
Similarly,
\begin{equation}\label{eq8}
\Phi^{J,k}(\mathbf b_r)\mathcal C_{k-1}([\mathbf b_{i_1}\ \dots\ \mathbf b_{i_{k-1}}])=
\begin{cases}
\vzero, &\text{if } r\in\{i_1, \dots, i_{k-1}\};\\
\pm \mathcal C_k(\mathbf B)_{[i_1, i_2, \dots, i_{k-1}, r]}, &\text{if }r\not \in\{i_1, \dots, i_{k-1}\}.
\end{cases}
\end{equation}
Now,  equations \eqref{eq6}--\eqref{eq8} yield
\begin{equation}
\label{eqd1dm-1}
\begin{split}
\sum_{\substack{1\leq i_1<\dots<i_{k-1}\leq R\\ i_1, \dots, i_{k-1}\ne r}}
 d_{(i_1,\dots, i_{k-1})}
&\left(
\mathcal C_k(\mathbf A)_{[i_1, i_2, \dots, i_{k-1}, r]}\otimes
\mathcal C_k(\mathbf B)_{[i_1, i_2, \dots, i_{k-1}, r]}\right)=\vzero,\\
&r=1, \dots, R.
\end{split}
\end{equation}
Since $\mathcal C_k(\mathbf A)\odot \mathcal C_k(\mathbf B)$ has full column rank,  it follows that for all $r=1, \dots, R$,
$$
d_{(i_1,\dots, i_{k-1})}=0, \text{  whenever   } 1\leq i_1<\dots<i_{k-1}\leq R\text{ and } i_1, \dots, i_{k-1}\ne r.
$$
It immediately follows   that ${\mathbf d}_{S_R^{k-1}}=\vzero$. Hence, $\mathcal C_{k-1}(\mathbf A)\odot \mathcal C_{k-1}(\mathbf B)$ has  full column rank.
\end{proof}

\begin{lemma}\label{compoundumuk}
Let $\mathbf A\in \mathbb F^{I\times R}$,  $\mathbf B\in \mathbb F^{J\times R}$,  and $1<m\leq \min(I, J)$. Then condition $\text{\textup{(U{\scriptsize m})}}$ implies conditions \textup{(U{\scriptsize m-1})},$\dots$,\textup{(U{\scriptsize 1})}.
\end{lemma}
\begin{proof}
It is sufficient to prove that $\text{\textup{(U{\scriptsize k})}}$ implies $\text{\textup{(U{\scriptsize k-1})}}$ for  $k\in\{m,m-1,\dots,2\}$.
 Assume to the contrary that $\text{\textup{(U{\scriptsize k-1})}}$ does not hold.
Then there exists a nonzero vector $\hatdSmR{k-1}{R}$
 such that
$\left[\mathcal C_{k-1}(\mathbf A)\odot \mathcal C_{k-1}(\mathbf B)\right]\hatdSmR{k-1}{R}=\vzero$.
Analogous to the  proof of  Lemma \ref{compoundkhr} we obtain that \eqref{eqd1dm-1} holds with
\begin{equation}\label{eqdi}
d_{(i_1,\dots, i_{k-1})}=d_{i_1}\cdots d_{i_{k-1}}, \qquad (i_1, \dots, i_{k-1})\in S^{k-1}_R.
\end{equation}
Thus,  multiplying the $r$-th equation from \eqref{eqd1dm-1} by $d_r$,
 for $1\le r\le R$, we obtain
\begin{equation}
\label{eq33}
\sum_{\substack{1\leq i_1<\dots<i_{k-1}\leq R\\ i_1, \dots, i_{k-1}\ne r}}
 d_{i_1}\cdots d_{i_{k-1}}d_r
\left(
\mathcal C_k(\mathbf A)_{[i_1, i_2, \dots, i_{k-1}, r]}\otimes
\mathcal C_k(\mathbf B)_{[i_1, i_2, \dots, i_{k-1}, r]}\right)=\vzero.
\end{equation}
Summation of \eqref{eq33} over $r$ yields
\begin{equation}
k\left[\mathcal C_k(\mathbf A)\odot \mathcal C_k(\mathbf B)\right]
\hatdSmR{k}{4}
=\vzero.
\label{eq5Uproof}
\end{equation}
Since $\text{\textup{(U{\scriptsize k})}}$ holds,  \eqref{eq5Uproof} implies that
\begin{equation*}
 d_{i_1}\cdots d_{i_k}=0, \qquad (i_1, \dots, i_k)\in S^k_R.
\end{equation*}
Since $\hatdSmR{k-1}{R}$ is nonzero, it follows
that exactly $k-1$ of the $R$ values $d_1, \dots, d_R$ are different from zero. Therefore,  $\hatdSmR{k-1}{R}$ has exactly one nonzero component.
It follows that the matrix $\mathcal C_{k-1}(\mathbf A)\odot \mathcal C_{k-1}(\mathbf B)$ has a zero column. Hence,
$\min(k_{\mathbf A}, k_{\mathbf B})\leq k-2$. On the other hand,  Lemma \ref{lemmaU_mKhR} implies that  $\min(k_{\mathbf A}, k_{\mathbf B})\geq k$, which is a contradiction.
\end{proof}

The following lemma completes scheme \eqref{maindiagramintro}.

\begin{lemma}\label{compoundkhrkrusk}
Let $\mathbf A\in \mathbb F^{I\times R}$,  $\mathbf B\in \mathbb F^{J\times R}$,  and $m\leq \min(I, J)$.
Then condition $\text{\textup{(K{\scriptsize m})}}$ implies condition \textup{(C{\scriptsize m})}.
\end{lemma}

\begin{proof}
We give the proof for the case  $r_{\mathbf A}+k_{\mathbf B}\geq R+m$ and $k_{\mathbf A}\geq m$; the case
$r_{\mathbf B}+k_{\mathbf A}\geq R+m$ and $k_{\mathbf B}\geq m$ follows by symmetry. We obviously have $k_{\mathbf B} \geq m$.

In the case $k_{\mathbf B}=m$, we have $r_{\mathbf A} =R$. Lemma \ref{LemmaCompound} (5) implies that the  $C^m_I\times C^m_{R}$ matrix $C_m(\mathbf A)$ has full column rank. The fact that $k_{\mathbf B}=m$  implies that every column of $C_m(\mathbf B)$ contains at least one nonzero entry. It immediately  follows that $\mathcal C_m(\mathbf A)\odot \mathcal C_m(\mathbf B)$ has  full column rank.

We now consider the case $k_{\mathbf B}>m$.
\begin{romannum}
\item
Suppose that
$
[\mathcal C_m(\mathbf A)\odot \mathcal C_m(\mathbf B)]{\mathbf d}_{S_R^m} =\vzero_{C^m_I C^m_J}
$
for some vector ${\mathbf d}_{S_R^m} \in\mathbb F^{C^m_R}$.
Then, by \eqref{eqmatrtovec_b},
\begin{equation}
\mathcal C_m(\mathbf A) \textup{\text{Diag}}({\mathbf d}_{S_R^m})\mathcal C_m(\mathbf B)^T=\mzero_{C^m_I \times C^m_J}.
\label{eqtfrom1.4}
\end{equation}
\item Let us for now assume that the last $r_{\mathbf A}$ columns of $\mathbf A$ are linearly independent. We show that
$d_{(k_{\mathbf B}-m+1, \dots, k_{\mathbf B})}=0$.

By definition of $k_{\mathbf B}$, the matrix $\mathbf X:=\left[\begin{matrix}\mathbf b_1&\dots&\mathbf b_{k_{\mathbf B}}\end{matrix}\right]^T$
has full row rank. Hence, $\mathbf X \mathbf X^\dagger =\mathbf  I_{k_{\mathbf B}}$
, where $\mathbf X^\dagger$ denotes a right inverse of $\mathbf X$.
Denoting
$$
\mathbf Y:=
\mathbf X^\dagger \left[\begin{array}{l}
\mzero_{(k_{\mathbf B}-m)\times m}\\
\mathbf I_m
\end{array}\right],
$$
we have
\begin{equation*}
\begin{split}
 \mathbf B^T\mathbf Y
 =
 &\begin{bmatrix}\mathbf X\\ \left[\begin{matrix}\mathbf b_{k_{\mathbf B}+1}&\dots&\mathbf b_{R}\end{matrix}\right]^T\end{bmatrix}\mathbf X^\dagger
 \left[\begin{array}{l}
\mzero_{(k_{\mathbf B}-m)\times m}\\
\mathbf I_m
\end{array}\right]
 \\
 =&\begin{bmatrix}\mathbf I_{k_{\mathbf B}}\\ \boxplus_{(R-k_{\mathbf B})\times k_{\mathbf B}}\end{bmatrix}
 \left[\begin{array}{l}
\mzero_{(k_{\mathbf B}-m)\times m}\\
\mathbf I_m
\end{array}\right]
 =
\left[\begin{array}{l}
\mzero_{(k_{\mathbf B}-m)\times m}\\
\mathbf I_m\\
\boxplus_{(R-k_{\mathbf B})\times m}
\end{array}\right],
\end{split}
\end{equation*}
where $\boxplus_{p\times q}$ denotes a $p\times q$ matrix that is not further specified.
From the definition of the $m$-th compound matrix it follows  that
\begin{equation}
\mathcal C_m(\mathbf B^T\mathbf Y)
=
\left[\begin{array}{l}
\vzero_{C^m_R-C^m_{R-k_{\mathbf B}+m}}\\
1\\
\boxplus_{(C^m_{R-k_{\mathbf B}+m}-1)\times 1}
\end{array}\right].
\label{eq:proofprop3.9}
\end{equation}
We now have
\begin{eqnarray}
\vzero_{C^m_I} & = &
\mzero_{C_I^m\times C_J^m} \cdot \mathcal C_m(\mathbf Y) \nonumber \\
& \stackrel{\eqref{eqtfrom1.4}}{=} &
\mathcal C_m(\mathbf A) \textup{\text{Diag}}({\mathbf d}_{S_R^m})\mathcal C_m(\mathbf B^T)\cdot\mathcal C_m(\mathbf Y)
\nonumber  \\
& = &\mathcal C_m(\mathbf A) \textup{\text{Diag}}({\mathbf d}_{S_R^m})\mathcal C_m(\mathbf B^T\mathbf Y) \nonumber  \\
& \stackrel{\eqref{eq:proofprop3.9}}{=} &
\mathcal C_m(\mathbf A)
\left[\begin{matrix}
\vzero_{C^m_R-C^m_{R-k_{\mathbf B}+m}} \\
d_{(k_{\mathbf B}-m+1, \dots, k_{\mathbf B})}\\
\boxplus_{(C^m_{R-k_{\mathbf B}+m}-1)\times 1}
\end{matrix}\right].
\label{compdependent}
\end{eqnarray}
Since the last $r_{\mathbf A}$ columns of $\mathbf A$ are linearly independent,  Lemma \ref{LemmaCompound} (5) implies that the  $C^m_I\times C^m_{r_{\mathbf A}}$
matrix  $\mathbf M=\mathcal C_m(\left[\begin{matrix}\mathbf a_{R-r_{\mathbf A}+1}&\dots&\mathbf a_R\end{matrix}\right])$ has full column rank. By definition, $\mathbf M$ consists of the last $C^m_{r_{\mathbf A}}$ columns of $C_m(\mathbf A)$. Obviously,  $r_{\mathbf A}+k_{\mathbf B}\geq R+m$ implies $C^m_{r_{\mathbf A}}\geq C^m_{R-k_{\mathbf B}+m}$.
Hence, the last $C^m_{R-k_{\mathbf B}+m}$ columns of $C_m(\mathbf A)$ are linearly independent and the coefficient vector in \eqref{compdependent} is zero. In particular,
$d_{(k_{\mathbf B-m+1},\dots, k_{\mathbf B})}=0$.

\item
We show that  $d_{(j_1,\dots, j_m)}=0$ for any choice of ${j_1}$, ${j_2}$,  \dots, ${j_m}$, $1\leq j_1<\dots <j_m\leq R$.

Since $k_{\mathbf A}\geq m$, the set of vectors $\mathbf a_{j_1},\dots,\mathbf a_{j_m}$  is linearly independent.
Let us extend the set $\mathbf a_{j_1},\dots,\mathbf a_{j_m}$ to a basis of $\textup{range}(\mathbf A)$
by adding $r_{\mathbf A}-m$ linearly independent columns of $\mathbf A$. Denote these basis vectors by $\mathbf a_{j_1},\dots,\mathbf a_{j_m}, \mathbf a_{j_{m+1}},\dots, \mathbf a_{j_{r_{\mathbf A}}}$.
It is clear that there exists an  $R\times R$ permutation matrix $\mathbf \Pi$  such that the
$(\mathbf A\mathbf \Pi)_{R-r_{\mathbf A}+1}=\mathbf a_{j_1},\dots, (\mathbf A\mathbf\Pi)_R=\mathbf a_{j_{r_{\mathbf A}}}$,
where here and in the sequel $(\mathbf A\mathbf \Pi)_r$ denotes the $r$-th column of the matrix $\mathbf A\mathbf \Pi$.
Moreover, since $k_{\mathbf B}-m+1\geq R-r_{\mathbf A}+1$ we can choose $\mathbf\Pi$ such that it additionally satisfies
$(\mathbf A\mathbf{\Pi})_{k_{\mathbf B}-m+1} = \mathbf a_{j_1}$, $(\mathbf A\mathbf{\Pi})_{k_{\mathbf B}-m+2} = \mathbf a_{j_2}$, \ldots, $(\mathbf A\mathbf{\Pi})_{k_{\mathbf B}} = \mathbf a_{j_m}$.
We can now reason as under (ii) for $\mathbf A \mathbf{\Pi}$ and $\mathbf B \mathbf{\Pi}$ to obtain that $d_{(j_1,\dots, j_m)}=0$.

\item From (iii) we immediately obtain that ${\mathbf d}_{S_R^m}=\vzero$.
Hence, $\mathcal C_m(\mathbf A)\odot \mathcal C_m(\mathbf B)$ has  full column rank.
\end{romannum}
\end{proof}
We now give results that concern $\text{\textup{(H{\scriptsize m})}}$.
\begin{lemma}\label{KmHm}
Let $\mathbf A\in \mathbb F^{I\times R}$,  $\mathbf B\in \mathbb F^{J\times R}$,  and $m\leq \min(I, J)$.
Then condition $\text{\textup{(K{\scriptsize m})}}$ implies condition \textup{(H{\scriptsize m})}.
\end{lemma}
\begin{proof}
We give the proof for the case  $r_{\mathbf A}+k_{\mathbf B}\geq R+m$ and $k_{\mathbf A}\geq m$; the case
$r_{\mathbf B}+k_{\mathbf A}\geq R+m$ and $k_{\mathbf B}\geq m$ follows by symmetry. We obviously have $k_{\mathbf B} \geq m$.
\begin{romannum}
\item Suppose that $\delta\leq m$. Then $r_{\tilde{\mathbf A}}=r_{\tilde{\mathbf B}}=\delta$. Hence, $H(\delta)=\delta$.
\item Suppose that $\delta\geq m$ and $\delta\geq k_{\mathbf B}$. Then $r_{\tilde{\mathbf B}}\geq k_{\mathbf B}$. Let $\tilde{\mathbf A}^c$ denote the $I\times(R-\delta)$  matrix obtained from $\mathbf A$ by removing the columns that are also in $\tilde{\mathbf A}$. Then $r_{\tilde{\mathbf A}}\geq r_{\mathbf A}-r_{\tilde{\mathbf A}^c}$.
    Hence, $r_{\tilde{\mathbf A}}+r_{\tilde{\mathbf B}}-\delta\geq r_{\mathbf A}-r_{\tilde{\mathbf A}^c}+r_{\tilde{\mathbf B}}-\delta\geq r_{\mathbf A}-(R-\delta)+k_{\mathbf B}-\delta\geq m$. Thus, $H(\delta)\geq m$.
    \item Suppose that $k_{\mathbf B}\geq\delta\geq m$. Then $r_{\tilde{\mathbf B}}=\delta$. Since
    $r_{\tilde{\mathbf A}}\geq\min(\delta,k_{\mathbf A})\geq m$, it follows that
  $r_{\tilde{\mathbf A}}+r_{\tilde{\mathbf B}}-\delta\geq m+\delta-\delta=m$. Thus, $H(\delta)\geq m$.
\end{romannum}
\end{proof}

\begin{lemma}\label{HmUm}
Let $\mathbf A\in \mathbb F^{I\times R}$,  $\mathbf B\in \mathbb F^{J\times R}$,  and $m\leq \min(I, J)$.
Then condition $\text{\textup{(H{\scriptsize m})}}$ implies condition \textup{(U{\scriptsize m})}.
\end{lemma}
\begin{proof}
The following proof is based on the proof of Rank Lemma from \cite[p. 121]{Kruskal1977}.
Let $(\mathcal C_{m}(\mathbf A )\odot\mathcal C_m(\mathbf B))\hatdSmR{m}{R}=\vzero$ for $\hatdSmR{m}{R}$
associated with $\mathbf d\in\mathbb F^R$.
By  Lemma \ref{cor2.6},
$\mathcal C_{m}(\mathbf B  \textup{\text{Diag}}(\mathbf d)\mathbf A^T)=\mathcal C_{m}(\mathbf B) \textup{\text{Diag}}(\hatdSmR{m}{R})\mathcal C_m(\mathbf A)^T=
\mzero$. Hence, $r_{\mathbf B  \textup{\text{Diag}}(\mathbf d)\mathbf A^T}\leq m-1$.
Let $\omega(\mathbf d)=\delta$ and $d_{i_1}=\dots=d_{i_{R-\delta}}=0$. Form the matrices $\tilde{\mathbf  A}$, $\tilde{\mathbf B}$ and the vector $\tilde{\mathbf d}$ by dropping the columns of $\mathbf A$, $\mathbf B$ and the entries of  $\mathbf d$ indexed by $i_1,\dots, i_{R-\delta}$.
From the   Sylvester rank  inequality we obtain
\begin{align*}
\min(\delta,m)\leq H(\delta)&\leq r_{\tilde{\mathbf B}}+r_{\tilde{\mathbf A}}-\delta=
r_{\tilde{\mathbf B}\textup{\text{Diag}}(\tilde{\mathbf d})}+r_{\textup{\text{Diag}}(\tilde{\mathbf d})\tilde{\mathbf A}}-r_{\textup{\text{Diag}}(\tilde{\mathbf d})}
\\
&\leq r_{\tilde{\mathbf B}  \textup{\text{Diag}}(\tilde{\mathbf d})\tilde{\mathbf A}^T}=
r_{\mathbf B  \textup{\text{Diag}}(\mathbf d)\mathbf A^T}\leq m-1.
\end{align*}
Hence, $\delta\leq m-1$. From  Lemma \ref{Lemma2.3} 1) it follows that $\hatdSmR{m}{R}=\vzero$.
\end{proof}

The remaining part of this section concerns \eqref{eq1.14}.

Lemma \ref{compoundumuk} and Lemma \ref{lemmaU_mKhR} can be summarized as follows:
$$
\text{\textup{(U{\scriptsize m})}}\Rightarrow
\left\{
\begin{array}{l}
\text{\textup{(U{\scriptsize k})}},\ k\leq m;\\
\mathbf A\odot\mathbf B\ \text{ has full column rank}\ (\Leftrightarrow \text{\textup{(U{\scriptsize 1})}});\\
\min (k_{\mathbf A},k_{\mathbf B})\geq m.
\end{array}
\right.
$$
The following example demonstrates that similar implications do not  necessarily hold for  $\text{\textup{(W{\scriptsize m})}}$.
Namely, in general, $\text{\textup{(W{\scriptsize m})}}$ does not imply any of the following conditions:
\begin{align*}
&\text{\textup{(W{\scriptsize k})}}\text{ for }\ k\leq m-1, \\
&\mathbf A\odot\mathbf B\ \text{ has full column rank},\\
&\min (k_{\mathbf A},k_{\mathbf B})\geq m-1.
\end{align*}
\begin{example}
Let $m=2$ and let
$$
\mathbf A=\left[\begin{matrix} 1&0&0&1\\ 0&1&0&1\end{matrix}\right],\quad
\mathbf B=\left[\begin{matrix} 1&0&1&1\\ 0&1&1&2\end{matrix}\right],\quad
\mathbf C=\left[\begin{matrix} 0&0&1&0\\ 1&1&0&1\end{matrix}\right].
$$
Let us show that condition $\text{\textup{(W{\scriptsize 2})}}$ holds but condition $\text{\textup{(W{\scriptsize 1})}}$ does not hold.
It is easy to check that
$$
\mathbf A\odot\mathbf B= \left[\begin{matrix} 1&0&0&1\\ 0&0&0&2\\ 0&0&0&1\\ 0&1&0&2\end{matrix}\right],\quad
\mathcal C_2(\mathbf A)\odot \mathcal C_2(\mathbf B)=\left[\begin{matrix} 1&0&2&0&1&0\end{matrix}\right].
$$
Let $\mathbf d\in\textup{range}(\mathbf C^T)$. Then there exist $x_1,x_2\in\mathbb F$ such that
$\mathbf d=\left[\begin{matrix} x_2&x_2&x_1&x_2\end{matrix}\right]^T$. Hence,
$\hatdSmR{2}{4}=
\left[\begin{matrix} x_2^2&x_1x_2&x_2^2&x_1x_2&x_2^2&x_1x_2\end{matrix}\right]^T$. Therefore
\begin{align*}
(\mathcal C_2(\mathbf A)\odot \mathcal C_2(\mathbf B))\hatdSmR{2}{4}=\vzero\Leftrightarrow x_2=0\Rightarrow \omega(\mathbf d)\leq 1 = 2-1\Rightarrow
&\text{\textup{(W{\scriptsize 2})}} \text{ holds},\\
(\mathbf A\odot \mathbf B)\mathbf e_3^4=\vzero,\ \ \mathbf e_3^4\in \textup{range}(\mathbf C^T),\ \ \omega(\mathbf e_3^4)=1>1-1
\Rightarrow&\text{\textup{(W{\scriptsize 1})}} \text{ does not hold},
\end{align*}
where $\mathbf e_3^4=\left[\begin{matrix}0&0&1&0\end{matrix}\right]^T$.
In particular,  $\mathbf A\odot\mathbf B$ does not have full column rank.
Besides, since the matrix $\mathbf A$ has a zero column, it follows that  $\min(k_{\mathbf A},k_{\mathbf B})=0<m-1$.
\end{example}

The following lemma now establishes \eqref{eq1.14}.
\begin{lemma}\label{compoundwmwk}
Let $\mathbf A\in \mathbb F^{I\times R}$,  $\mathbf B\in \mathbb F^{J\times R}$, $\mathbf C\in \mathbb F^{K\times R}$, $1<m\leq \min(I, J)$,
and $\min (k_{\mathbf A},k_{\mathbf B})\geq m-1$.
Then condition $\text{\textup{(W{\scriptsize m})}}$ implies conditions \textup{(W{\scriptsize m-1})},$\dots$,\textup{(W{\scriptsize 1})}.
\end{lemma}
\begin{proof}
The proof is the same as the proof of Lemma \ref{compoundumuk}, with the difference that instead of Lemma \ref{lemmaU_mKhR}
we use the condition $\min (k_{\mathbf A},k_{\mathbf B})\geq m-1$.
\qquad\end{proof}

\section{Sufficient conditions for the uniqueness of one factor matrix}\label{Section4}

In this section we establish conditions under which a PD is canonical, with one of  the factor matrices unique.
We have the following formal definition.
\begin{definition}
Let $\mathcal T$ be a tensor of rank $R$. The  first (resp. second or third) factor matrix of  $\mathcal T$ is {\em  unique}
if $\mathcal T=[\mathbf A,\mathbf B,\mathbf C]_R=[\bar{\mathbf A},\bar{\mathbf B},\bar{\mathbf C}]_R$ implies that there exist an $R\times R$ permutation matrix
$\mathbf{\Pi}$ and an $R\times R$ nonsingular diagonal matrix ${\mathbf \Lambda}_{\mathbf A}$ (resp. ${\mathbf \Lambda}_{\mathbf B}$ or ${\mathbf \Lambda}_{\mathbf C}$) such that
$
\bar{\mathbf A}=\mathbf A\mathbf{\Pi}{\mathbf \Lambda}_{\mathbf A} \quad (\text{resp.} \quad
\bar{\mathbf B}=\mathbf B\mathbf{\Pi}{\mathbf \Lambda}_{\mathbf B} \quad \ \text{ or }\ \
\bar{\mathbf C}=\mathbf C\mathbf{\Pi}{\mathbf \Lambda}_{\mathbf C}).
$
\end{definition}
\subsection{Conditions  based on  (U{\scriptsize m}), (C{\scriptsize m}), (H{\scriptsize m}), and (K{\scriptsize m})}
\label{subsect:detUmCmWm}
First, we recall Kruskal's permutation lemma, which we will use in the proof of  Proposition \ref{prmostgeneral}.
\begin{lemma}\label{lemmapermutation}
\cite{Kruskal1977,JiangSid2004, Stegeman2007}
Consider two matrices $\bar{\mathbf C}\in\mathbb F^{K\times \bar{R}}$ and $\mathbf C\in\mathbb F^{K\times R}$ such that $\bar{R}\leq R$ and $k_{\mathbf C}\geq 1$.
If for every vector $\mathbf x$ such that $\omega(\bar{\mathbf C}^T\mathbf x)\leq \bar{R}-r_{\bar{\mathbf C}}+1$,
 we have $\omega(\mathbf C^T\mathbf x)\leq \omega(\bar{\mathbf C}^T\mathbf x)$,  then $\bar{R}=R$ and there exist a unique permutation matrix $\mathbf{\Pi}$
and a unique nonsingular diagonal matrix ${\mathbf \Lambda}$ such that $\bar{\mathbf C}=\mathbf C\mathbf{\Pi}{\mathbf \Lambda}$.
\end{lemma}

We start the derivation of \eqref{manyimplicationsm} with the proof of the following proposition.
\begin{proposition}\label{prmostgeneral}
Let $\mathbf A\in \mathbb F^{I\times R}$,  $\mathbf B\in \mathbb F^{J\times R}$,  $\mathbf C\in \mathbb F^{K\times R}$, and let
$\mathcal T=[\mathbf A, \mathbf B, \mathbf C]_R$.
Assume that
\begin{romannum}
\item $k_{\mathbf C}\geq 1$;
\item
$m=R-r_{\mathbf C}+2\leq\min(I, J)$;
\item condition $\text{\textup{(U{\scriptsize m})}}$ holds.
\end{romannum}
Then $r_{\mathcal T}=R$ and the  third factor matrix of $\mathcal T$  is unique.
\end{proposition}
\begin{proof}
Let $\mathcal T=[\bar{\mathbf A}, \bar{\mathbf B}, \bar{\mathbf C}]_{\bar R}$ be a CPD of $\mathcal T$, which implies $\bar{R}\leq R$. We have $(\mathbf A\odot\mathbf B)\mathbf C^T=(\bar{\mathbf A}\odot\bar{\mathbf B})\bar{\mathbf C}^T$. We check that the conditions of Lemma \ref{lemmapermutation} are satisfied. From Lemma \ref{compoundumuk} it follows that conditions $\text{\textup{(U{\scriptsize  m-1})}}, \dots, \text{\textup{(U{\scriptsize 2})}}, \text{\textup{(U{\scriptsize 1})}}$ hold.
The fact that $\text{\textup{(U{\scriptsize 1})}}$ holds, means that
$\mathbf A\odot\mathbf B$ has full column rank. Hence,
\begin{equation}\label{star}
r_{\mathbf C}=r_{\mathbf C^T}=r_{(\mathbf A\odot\mathbf B)\mathbf C^T}=
r_{(\bar{\mathbf A}\odot\bar{\mathbf B})\bar{\mathbf C}^T}\leq
r_{\bar{\mathbf C}^T}=r_{\bar{\mathbf C}}.
\end{equation}
Consider any vector $\mathbf x\in\mathbb F^K$ such that
$$
\omega(\bar{\mathbf C}^T\mathbf x) := k-1\leq \bar{R}-r_{\bar{\mathbf C}}+1,
$$
as in Lemma \ref{lemmapermutation}.
Then $r_{\bar{\mathbf A}\textup{\text{Diag}}(\bar{\mathbf C}^T\mathbf x)\bar{\mathbf B}^T}\leq k-1$ and, by \eqref{star},
$$
\bar{R}-r_{\bar{\mathbf C}}+1\leq R-r_{\mathbf C}+1=m-1,
$$
which implies $k \leq m$. We have
$(\mathbf A\odot\mathbf B)\mathbf C^T\mathbf x=(\bar{\mathbf A}\odot\bar{\mathbf B})\bar{\mathbf C}^T\mathbf x$ so
$
\mathbf A\textup{\text{Diag}}(\mathbf C^T\mathbf x)\mathbf B^T=
\bar{\mathbf A}\textup{\text{Diag}}(\bar{\mathbf C}^T\mathbf x)\bar{\mathbf B}^T.
$
From Lemma \ref{LemmaCompound}(1) it follows
\begin{eqnarray*}
\mathcal C_k( \mathbf A\textup{\text{Diag}}(\mathbf C^T\mathbf x)\mathbf B^T  ) & = &
\mathcal C_k(\bar{\mathbf A}\textup{\text{Diag}}(\bar{\mathbf C}^T\mathbf x)\bar{\mathbf B}^T) \\
& = & \mathcal C_k(\bar{\mathbf A})\mathcal C_k(\textup{\text{Diag}}(\bar{\mathbf C}^T\mathbf x))\mathcal C_k(\bar{\mathbf B}^T) \\
& = &\mzero,
\end{eqnarray*}
in which the latter equality follows from Lemma \ref{Lemma2.3}(1).
Hence, by Lemma \ref{cor2.6},
$$
(\mathcal C_{k}(\mathbf A )\odot\mathcal C_k(\mathbf B))\hatdSmR{k}{R}=\vzero
$$
for $\mathbf d:=\mathbf C^T\mathbf x\in\mathbb F^R$.
Since condition $\text{\textup{(U{\scriptsize k})}}$ holds for $\mathbf A$ and $\mathbf B$,  it follows that
$\omega(\mathbf C^T\mathbf x)\leq k-1=\omega(\bar{\mathbf C}^T\mathbf x)$. Hence, by  Lemma \ref{lemmapermutation},
$\bar{R}=R$ and the matrices $\mathbf C$ and $\bar{\mathbf C}$ are the same up to permutation and column scaling.
\qquad\end{proof}

The implications $\text{\textup{(C{\scriptsize m})}} \Rightarrow \text{\textup{(U{\scriptsize m})}}$ and
$\text{\textup{(H{\scriptsize m})}} \Rightarrow \text{\textup{(U{\scriptsize m})}}$
in scheme \eqref{maindiagramintro} lead to Corollary \ref{corr111} and to Theorem \ref{theoremKruskalnew3}, respectively.
The implication
$\text{\textup{(K{\scriptsize m})}} \Rightarrow \text{\textup{(C{\scriptsize m})}}$ in turn leads to Corollary \ref{corrpukr}. Clearly, conditions $\text{\textup{(C{\scriptsize m})}}$, $\text{\textup{(H{\scriptsize m})}}$, and $\text{\textup{(K{\scriptsize m})}}$ are more restrictive than $\text{\textup{(U{\scriptsize m})}}$. On the other hand, they may be easier to verify.

\begin{corollary}\label{corr111}
Let $\mathbf A\in \mathbb F^{I\times R}$,  $\mathbf B\in \mathbb F^{J\times R}$,  $\mathbf C\in \mathbb F^{K\times R}$, and let $\mathcal T=[\mathbf A, \mathbf B, \mathbf C]_R$. Assume that
\begin{romannum}
\item $k_{\mathbf C}\geq 1$;
\item
$m=R-r_{\mathbf C}+2\leq\min(I, J)$;
\item
$\mathcal C_{m}(\mathbf A)\odot \mathcal C_{m}(\mathbf B) \text{ has full column rank}$.\hfill \text{\textup{(C{\scriptsize m})}}
\end{romannum}
Then $r_{\mathcal T}=R$ and the third factor matrix of  $\mathcal T$  is unique.
\end{corollary}

\begin{corollary}\label{corrpukr}
Let $\mathbf A\in \mathbb F^{I\times R}$,  $\mathbf B\in \mathbb F^{J\times R}$,  $\mathbf C\in \mathbb F^{K\times R}$, and
$\mathcal T=[\mathbf A, \mathbf B, \mathbf C]_R$.
Let also $m:=R-r_{\mathbf C}+2$. If
\begin{equation}
\label{eqlongform}
\left\{
  \begin{array}{rl}
    r_{\mathbf A}+k_{\mathbf B}&\geq R+m,\\
    k_{\mathbf A}&\geq m,\\
     k_{\mathbf C}&\geq 1
  \end{array}
\right.
\qquad\text{ or }\qquad
\left\{
  \begin{array}{rl}
      r_{\mathbf B}+k_{\mathbf A}&\geq R+m,\\
      k_{\mathbf B}&\geq m,\\
      k_{\mathbf C}&\geq 1,
  \end{array}
\right.
\end{equation}
then $r_{\mathcal T}=R$ and the third  factor matrix of $\mathcal T$  is unique.
\end{corollary}

\begin{remark}
Condition $\textup{(ii)}$ in  Proposition \ref{prmostgeneral} and Corollary \ref{corr111} guarantees that the matrices $\mathcal C_{m}(\mathbf A)$ and $\mathcal C_{m}(\mathbf B)$ are defined.
In Corollary \ref{corrpukr}, \textup{(K{\scriptsize m})} cannot hold if $m=R-r_{\mathbf C}+2 > \min(I, J)$.
\end{remark}

\begin{corollary}\label{corrpukrnext}
Let $\mathbf A\in \mathbb F^{I\times R}$,  $\mathbf B\in \mathbb F^{J\times R}$,  $\mathbf C\in \mathbb F^{K\times R}$, and
$\mathcal T=[\mathbf A, \mathbf B, \mathbf C]_R$.
Let also $m:=R-r_{\mathbf C}+2$. If
\begin{equation}
\label{eqshortform}
\begin{cases}
k_{\mathbf C}\geq 1,\\
\min(k_{\mathbf A}, k_{\mathbf B})\geq m,\\
\max(r_{\mathbf A}+k_{\mathbf B},  r_{\mathbf B}+k_{\mathbf A})\geq R+m,
\end{cases}
\end{equation}
then $r_{\mathcal T}=R$ and the third factor matrix  of $\mathcal T$  is unique.
\end{corollary}
\begin{proof}
It can easily be checked that \eqref{eqshortform} and \eqref{eqlongform} are equivalent.
\end{proof}
\begin{remark}\label{remark_unimode}
It is easy to see that Corollaries \ref{corrpukr} and \ref{corrpukrnext} are equivalent to Theorem \ref{theorem_uni-mode}. Indeed,
if $m=R-r_{\mathbf C}+2$, then
\begin{align*}
\eqref{unique_one_new_paper}\Leftrightarrow
\begin{cases}
k_{\mathbf C}\geq 1,\\
\min(k_{\mathbf A}, k_{\mathbf B})\geq m,\\
k_{\mathbf A}+k_{\mathbf B}+\max(r_{\mathbf A}-k_{\mathbf A}, r_{\mathbf B}-k_{\mathbf B})\geq R+m
\end{cases}\Leftrightarrow \eqref{eqshortform}\Leftrightarrow \eqref{eqlongform}.
\end{align*}
\end{remark}

\subsection{Conditions  based on (W{\scriptsize m})}\label{Appendix A}

In this subsection  we deal with condition
$\text{\textup{(W{\scriptsize m})}}$. Similar to condition $\text{\textup{(U{\scriptsize m})}}$ in Proposition \ref{prmostgeneral}, condition $\text{\textup{(W{\scriptsize m})}}$ will in Proposition \ref{prmostgeneraldis} imply the uniqueness of one factor matrix. However,  condition $\text{\textup{(W{\scriptsize m})}}$ is more relaxed than condition $\text{\textup{(U{\scriptsize m})}}$.
Like condition $\text{\textup{(U{\scriptsize m})}}$, condition $\text{\textup{(W{\scriptsize m})}}$ may be hard to check. We give an example in which the uniqueness of one factor matrix can nevertheless be demonstrated using condition $\text{\textup{(W{\scriptsize m})}}$.

\begin{proposition}\label{prmostgeneraldis}
Let $\mathbf A\in \mathbb F^{I\times R}$,  $\mathbf B\in \mathbb F^{J\times R}$,  $\mathbf C\in \mathbb F^{K\times R}$, and let
$\mathcal T=[\mathbf A, \mathbf B, \mathbf C]_R$.
Assume that
\begin{romannum}
\item $k_{\mathbf C}\geq 1$;
\item
$m=R-r_{\mathbf C}+2\leq\min(I, J)$;
\item $\mathbf A\odot\mathbf B$ has full column rank;
\item
conditions $\text{\textup{(W{\scriptsize m})}},\dots,\text{\textup{(W{\scriptsize 1})}}$ hold.
\end{romannum}
Then $r_{\mathcal T}=R$ and the third factor matrix  of $\mathcal T$  is unique.
\end{proposition}

\begin{proof}
The proof is analogous to the proof of Proposition \ref{prmostgeneral}, with two points of difference. Namely, the fact that $\mathbf A\odot\mathbf B$ has full column rank does not follow from $\text{\textup{(W{\scriptsize 1})}}$ but is assumed in condition (iii). Second, $\omega(\mathbf C^T\mathbf x)\leq \omega(\bar{\mathbf C}^T\mathbf x)$ follows from $\text{\textup{(W{\scriptsize k})}}$ instead of $\text{\textup{(U{\scriptsize k})}}$.\qquad
\end{proof}

From Lemmas \ref{compoundumuk} and \ref{PropositionA1} it follows that Proposition \ref{prmostgeneraldis} is more relaxed than Proposition \ref{prmostgeneral}.

Combining Proposition \ref{prmostgeneraldis} and Lemma \ref{compoundwmwk} we obtain the following result, which completes the derivation of
scheme \eqref{manyimplicationsm}.
\begin{corollary}\label{Corrolary4.9}
Let $\mathbf A\in \mathbb F^{I\times R}$,  $\mathbf B\in \mathbb F^{J\times R}$,  $\mathbf C\in \mathbb F^{K\times R}$, and let
$\mathcal T=[\mathbf A, \mathbf B, \mathbf C]_R$.
Assume that
\begin{romannum}
\item $k_{\mathbf C}\geq 1$;
\item
$m=R-r_{\mathbf C}+2\leq\min(I, J)$;
\item $\min (k_{\mathbf A},k_{\mathbf B})\geq m-1$;
\item $\mathbf A\odot\mathbf B$ has full column rank;
\item
condition $\text{\textup{(W{\scriptsize m})}}$ holds.
\end{romannum}
Then $r_{\mathcal T}=R$ and the third factor matrix of $\mathcal T$  is unique.
\end{corollary}
\begin{example}
Let $\mathcal T=[\mathbf A,\mathbf B,\mathbf C]_7$, with
\begin{align*}
&\mathbf A=
\left[
\begin{matrix}
1&	1&	0&	0&	0&	0&	0\\
1&	0&	1&	0&	0&	0&	0\\
1&	0&	0&	1&	0&	0&	0\\
1&	0&	0&	0&	1&	0&	0\\
0&	0&	0&	0&	0&	1&	0\\
0&	0&	0&	0&	0&	0&	1
\end{matrix}
\right],\qquad
\mathbf B=
\left[
\begin{matrix}
0&	1&	0&	0&	0&	0&	0\\
0&	0&	1&	0&	0&	0&	0\\
1&	0&	0&	1&	0&	0&	0\\
1&	0&	0&	0&	1&	0&	0\\
1&	0&	0&	0&	0&	1&	0\\
1&	0&	0&	0&	0&	0&	1
\end{matrix}
\right],\\
&\mathbf C=
\left[
\begin{matrix}
1&	0&	0&	1&	0&	0&	0\\
0&	1&	0&	0&	1&	0&	0\\
0&	0&	1&	0&	0&	1&	0\\
1&	0&	0&	0&	0&	0&	1
\end{matrix}
\right].
\end{align*}
We have
$$
k_{\mathbf A}=k_{\mathbf B}=4,\quad r_{\mathbf A}=r_{\mathbf B}=6,\quad k_{\mathbf C}=1,\quad r_{\mathbf C}=4,\quad m=5.
$$
Since $\min(k_{\mathbf A},k_{\mathbf B})<m$, it follows from Lemma \ref{lemmaU_mKhR} that  condition $\text{\textup{(U{\scriptsize m})}}$ does not hold.
We show that, on the other hand, condition $\text{\textup{(W{\scriptsize m})}}$  does  hold.
One can easily check that the rank of the $36\times 21$  matrix $\mathbf U=\mathcal C_{5}(\mathbf A)\odot \mathcal C_5(\mathbf B)$
is equal to $19$. Obviously, both the $(1,2,3,4,5)$-th  and the $(1,4,5,6,7)$-th  column of $\mathbf U$ are equal to zero.
Hence, if $\mathbf U\hatdSmR{5}{7}=\vzero$ for $\mathbf d=\left[\begin{matrix} d_1&\dots&d_7\end{matrix}\right]$,
then at most the  two entries $d_1d_2d_3d_4d_5$ and $d_1d_4d_5d_6d_7$  of the vector $\hatdSmR{5}{7}$ are nonzero. Consequently, for a nonzero vector $\hatdSmR{5}{7}$ we have
\begin{equation}\label{systemdisc}
\left\{
\begin{array}{l}
d_2=d_3=0,\\
d_1d_4d_5d_6d_7\ne 0
\end{array}
\right.
\quad\text{ or }\quad
\left\{
\begin{array}{l}
d_6=d_7=0,\\
 d_1d_2d_3d_4d_5\ne 0.
 \end{array}
\right.
\end{equation}
On the other hand, since $\mathbf d\in\textup{range}(\mathbf C^T)$, there exists $\mathbf x\in \mathbb F^4$ such that
\begin{equation}\label{solutiondisc}
\mathbf d=\mathbf C^T\mathbf x=\left[\begin{matrix}x_1+x_4& x_2&x_3&x_1&x_2&x_3&x_4\end{matrix}\right].
\end{equation}
One can easily check that set \eqref{systemdisc} does not have solutions of the form \eqref{solutiondisc}.
Thus,  condition $\text{\textup{(W{\scriptsize 5})}}$ holds.

Corollary \ref{k-rank Khatri-Rao lemma} implies that $\mathbf A\odot \mathbf B$ has full column rank. Thus, by Corollary \ref{Corrolary4.9},
 $r_{\mathcal T}=7$ and the third factor  matrix of $\mathcal T$  is unique.

Note that, since $k_{\mathbf C}=1$, it follows from Theorem \ref{NecessaryCPD} that the CPD $\mathcal T=[\mathbf A,\mathbf B,\mathbf C]_7$ is not unique.
\end{example}

\section{Overall CPD uniqueness} \label{Section5}
\subsection{At least one factor matrix has full column rank}

The results for the case $r_{\mathbf C} = R$ are well-studied. They are summarized in \eqref{manyimplications2}. In particular, Theorems \ref{th:1.16}, \ref{Theorem1.12} and \ref{TheoremAlwin} present (\text{U{\scriptsize 2}}), (\text{C{\scriptsize 2}}) and (\text{K{\scriptsize 2}}), respectively, as sufficient conditions for CPD uniqueness.

The implications \textup{(i)}$\Leftrightarrow$\textup{(ii)} $\Leftrightarrow$\textup{(iii)} in Theorem \ref{th:1.16} were proved in \cite{JiangSid2004}. The core implication is \textup{(ii)}$\Rightarrow$\textup{(iii)}. This implication follows almost immediately from
Proposition \ref{prmostgeneral}, which establishes uniqueness of ${\mathbf C}$, as we show below.
Implication \textup{(iii)}$\Rightarrow$\textup{(i)} follows from  Theorem \ref{necessityU2}.
Together with an explanation of the equivalence \textup{(i)}$\Leftrightarrow$\textup{(ii)} we obtain a short proof of Theorem \ref{th:1.16}.

Next, Theorem \ref{Theorem1.12} follows immediately from Theorem \ref{th:1.16}. Theorem \ref{TheoremAlwin} in turn follows immediately from Theorem \ref{Theorem1.12}, cf. scheme \eqref{maindiagramintro}.

{\em Proof of Theorem \ref{th:1.16}.}
\textup{(i)}$\Leftrightarrow$\textup{(ii):}
Follows from Lemma \ref{lemma:equivUm} for $m=2$.

\textup{(ii)}$\Rightarrow$\textup{(iii):}  By Proposition \ref{prmostgeneral},   $r_{\mathcal T}=R$ and  the third factor matrix of $\mathcal T$ is unique.
That is, for any CPD $\mathcal T=[\bar{\mathbf A}, \bar{\mathbf B}, \bar{\mathbf C}]_R$ there exists
 a permutation matrix $\mathbf{\Pi}$ and a nonsingular diagonal matrix $\mathbf{\Lambda}_{\mathbf C}$ such that
 $\bar{\mathbf C}=\mathbf C\mathbf{\Pi}_{\mathbf C}\mathbf{\Lambda}_{\mathbf C}$.

Then, by \eqref{eqT_V},
 $(\mathbf A\odot\mathbf B)\mathbf C^T=(\bar{\mathbf A}\odot\bar{\mathbf B})\bar{\mathbf C}^T=
 (\bar{\mathbf A}\odot\bar{\mathbf B})\mathbf{\Lambda}_{\mathbf C}\mathbf{\Pi}^T_{\mathbf C}\mathbf C^T$.
Since the matrix $\mathbf C$ has full column rank, it follows that $\mathbf A\odot\mathbf B=(\bar{\mathbf A}\odot\bar{\mathbf B})\mathbf{\Lambda}_{\mathbf C}\mathbf{\Pi}^T_{\mathbf C}$. Equating columns, we obtain that there exist nonsingular diagonal matrices $\mathbf{\Lambda}_{\mathbf A}$ and $\mathbf{\Lambda}_{\mathbf B}$ such that $\bar{\mathbf A}=\mathbf A\mathbf{\Pi}_{\mathbf C}\mathbf{\Lambda}_{\mathbf A}$ and $\bar{\mathbf B}=\mathbf B\mathbf{\Pi}_{\mathbf C}\mathbf{\Lambda}_{\mathbf B}$, with $\mathbf{\Lambda}_{\mathbf A}\mathbf{\Lambda}_{\mathbf B}\mathbf{\Lambda}_{\mathbf C}=\mathbf I_R$. Hence, the
CPD of $\mathcal T$ is unique.

\textup{(iii)}$\Rightarrow$\textup{(i):} follows from Theorem \ref{necessityU2}. \qquad\endproof

{\em Proof of Theorem \ref{Theorem1.12}:}
Condition $\textup{(C{\scriptsize 2})}$ in Theorem \ref{Theorem1.12} trivially implies condition $\textup{(U{\scriptsize 2})}$ in Theorem \ref{th:1.16}. Hence, $r_{\mathcal T}=R$ and the CPD of $\mathcal T$ is unique.
\qquad\endproof

{\em Proof of Theorem \ref{TheoremAlwin}:} By Lemma  \ref{compoundkhrkrusk}, condition $\textup{(K{\scriptsize 2})}$ in Theorem \ref{TheoremAlwin} implies condition
$\textup{(C{\scriptsize 2})}$. Hence, by Theorem \ref{Theorem1.12}, $r_{\mathcal T}=R$ and the CPD of $\mathcal T$ is unique.
\qquad\endproof

\begin{remark}
The results obtained in \cite{TenBerge2002} (see the beginning of subsection \ref{subsubsectionnecessity}) can be completed as follows.
Let $\mathcal T=[\mathbf A,\mathbf B,\mathbf C]_R$, $2\leq r_{\mathcal T}=R\leq 4$. Assume without loss of generality that
$r_{\mathbf C}\geq\max(r_{\mathbf A},r_{\mathbf B})$.
For such low tensor rank, we have now a uniqueness condition that is both necessary and sufficient. That condition is that
$r_{\mathbf C}=R$ and \cond{U}{2} holds. Also, for
$R\leq 3$,  \cond{K}{2}, \cond{H}{2}, \cond{C}{2}, and \cond{U}{2} are equivalent.
For these values of $R$, condition \cond{K}{2} is the easiest one to check. For $R=4$,
\cond{H}{2}, \cond{C}{2}, and \cond{U}{2} are equivalent. The proofs are based on a check of all possibilities and are therefore omitted.
\end{remark}

\subsection{No factor matrix is required to have full column rank} \label{subsect:overallm}
Kruskal's original proof (also the simplified version in \cite{Stegeman2007}) of Theorem \ref{theoremKruskal} consists of three main steps. The first step is the proof of the permutation lemma (Lemma \ref{lemmapermutation}). The second and the third step concern the following two implications:
\begin{equation}\label{eq5.1}
\begin{split}
&k_{\mathbf A}+k_{\mathbf B}+k_{\mathbf C}\geq 2R+2\\
&\Rightarrow
\begin{cases}
k_{\mathbf A}+k_{\mathbf B}+k_{\mathbf C}\geq 2R+2,\\
r_{\mathcal T}=R,\\
\text{every factor matrix in}\ \mathcal T=[\mathbf A,\mathbf B,\mathbf C]_R\ \text{is by itself unique}
\end{cases}\\
&\Rightarrow\text{the overall CPD } \mathcal T=[\mathbf A,\mathbf B,\mathbf C]_R \text{ is unique.}
\end{split}
\end{equation}
That is, the proof goes via demonstrating that individual factor matrices are unique.
Similarly, the proof of uniqueness result \textup{(ii)}$\Leftrightarrow$\textup{(iii)} in Theorem \ref{th:1.16} for the case $r_{\mathbf C} = R$, goes, via Proposition \ref{prmostgeneral}, in two steps, which correspond to the proofs of the following two equivalences:
\begin{equation}\label{eq5.2}
\begin{split}
&(\textup{U{\scriptsize 2}})\Leftrightarrow
\begin{cases}
r_{\mathcal T}=R,\\
\text{the third factor matrix of}\ \mathcal T=[\mathbf A,\mathbf B,\mathbf C]_R\ \text{ is unique}
\end{cases}\\
&\Leftrightarrow \text{ the CPD of  }\mathcal T \textup{ is unique.}
\end{split}
\end{equation}
Again the proof goes via demonstrating that one factor matrix is unique.
Note that the second equivalence in \eqref{eq5.2} is almost immediate  since $\mathbf C$ has full column rank. In contrast, the proof of the second implication in \eqref{eq5.1}  is not trivial.

Scheme \eqref{manyimplicationsm} generalizes the first implications  in \eqref{eq5.1} and \eqref{eq5.2}. What remains for the demonstration of overall CPD uniqueness, is the generalization of the second implications.
This problem is addressed in Part II \cite{partII}. Namely, part of the discussion in  \cite{partII} is about investigating how, in cases where possibly none of the factor matrices has full column rank, uniqueness of one or more factor matrices implies overall uniqueness.

\section{Conclusion} \label{sect:concl}

We have given an overview of  conditions guaranteeing the \\
uniqueness of one factor matrix in a PD or uniqueness of an overall CPD. We have discussed properties of compound matrices and used them to build the schemes of implications \eqref{maindiagramintro} and \eqref{eq1.14}. For the case $r_{\mathbf C} = R$ we have demonstrated the overall CPD uniqueness results in \eqref{manyimplications2} using second compound matrices. Using \eqref{maindiagramintro} and \eqref{eq1.14} we have obtained relaxed conditions guaranteeing the uniqueness of one factor matrix, for instance ${\mathbf C}$. The general idea is to the relax the condition on ${\mathbf C}$, no longer requiring that it has full column rank, while making the conditions on ${\mathbf A}$ and ${\mathbf B}$ more restrictive. The latter are conditions on the Khatri-Rao product of $m$-th compound matrices of ${\mathbf A}$ and ${\mathbf B}$, where $m > 2$.  In Part II \cite{partII} we will use the results to derive relaxed conditions guaranteeing the uniqueness of the overall CPD.
\section{Acknowledgments}
 The authors would like to thank the anonymous reviewers
for their valuable comments and their suggestions to improve
the presentation of the paper. The authors are  also grateful for useful  suggestions from
Professor A. Stegeman (University of Groningen, The Netherlands).
\bibliographystyle{siam}
\bibliography{PartI}

\begin{thebibliography}{10}

\bibitem{1970_Carroll_Chang}
{\sc J.~Carroll and J.-J. Chang}, {\em {A}nalysis of individual differences in
  multidimensional scaling via an {N}-way generalization of "{E}ckart-{Y}oung"
  decomposition}, Psychometrika, 35 (1970), pp.~283--319.

\bibitem{Cichocki2009}
{\sc A.~Cichocki, R.~Zdunek, A.~H. Phan, and S.~Amari}, {\em {N}onnegative
  {M}atrix and {T}ensor {F}actorizations - {A}pplications to {E}xploratory
  {M}ulti-way {D}ata {A}nalysis and {B}lind {S}ource {S}eparation.}, Wiley,
  2009.

\bibitem{ComoJ10}
{\sc P.~Comon and C.~Jutten}, eds., {\em {H}andbook of {B}lind {S}ource
  {S}eparation, {I}ndependent {C}omponent {A}nalysis and {A}pplications},
  Academic Press, Oxford UK, Burlington USA, 2010.

\bibitem{2009Comonetall}
{\sc P.~Comon, X.~Luciani, and A.~L.~F. de~Almeida}, {\em Tensor
  decompositions, alternating least squares and other tales}, J. Chemometrics,
  23 (2009), pp.~393--405.

\bibitem{DeLathauwer2006}
{\sc L.~De~Lathauwer}, {\em A {L}ink {B}etween the {C}anonical {D}ecomposition
  in {M}ultilinear {A}lgebra and {S}imultaneous {M}atrix {D}iagonalization},
  SIAM J. Matrix Anal. Appl., 28 (2006), pp.~642--666.

\bibitem{LievenLL1}
\leavevmode\vrule height 2pt depth -1.6pt width 23pt, {\em Blind separation of
  exponential polynomials and the decomposition of a tensor in
  rank--$({L}_r,{L}_r,1)$ terms}, SIAM J. Matrix Anal. Appl., 32 (2011),
  pp.~1451--1474.

\bibitem{Lieven_ISPA}
\leavevmode\vrule height 2pt depth -1.6pt width 23pt, {\em A short introduction
  to tensor-based methods for {F}actor {A}nalysis and {B}lind {S}ource
  {S}eparation}, in ISPA 2011: Proceedings of the 7th International Symposium
  on Image and Signal Processing and Analysis,  (2011), pp.~558--563.

\bibitem{partII}
{\sc I.~Domanov and L.~De~Lathauwer}, {\em {O}n the {U}niqueness of the
  {C}anonical {P}olyadic {D}ecomposition of third-order tensors --- {P}art
  {II}: {U}niqueness of the overall decomposition}, ESAT-SISTA Internal Report,
  12-72, Leuven, Belgium: Department of Electrical Engineering (ESAT), KU
  Leuven,  (2012).

\bibitem{GuoMironBrieStegeman}
{\sc X.~Guo, S.~Miron, D.~Brie, and A.~Stegeman}, {\em {U}ni-{M}ode and
  {P}artial {U}niqueness {C}onditions for {CANDECOMP}/{PARAFAC} of
  {T}hree-{W}ay {A}rrays with {L}inearly {D}ependent {L}oadings}, SIAM J.
  Matrix Anal. Appl., 33 (2012), pp.~111--129.

\bibitem{Guo_Miron_Brie_Zhu_Liao2011}
{\sc X.~Guo, S.~Miron, D.~Brie, S.~Zhu, and X.~Liao}, {\em A
  {CANDECOMP/PARAFAC} perspective on uniqueness of {DOA} estimation using a
  vector sensor array}, IEEE Trans. Signal Process., 59 (2011), pp.~3475--3481.

\bibitem{Harshman1970}
{\sc R.~A. Harshman}, {\em Foundations of the {PARAFAC} procedure: {M}odels and
  conditions for an "explanatory" multi-modal factor analysis}, UCLA Working
  Papers in Phonetics, 16 (1970), pp.~1--84.

\bibitem{Harshman1972}
\leavevmode\vrule height 2pt depth -1.6pt width 23pt, {\em Determination and
  {P}roof of {M}inimum {U}niqueness {C}onditions for{ PARAFAC}1}, UCLA Working
  Papers in Phonetics, 22 (1972), pp.~111--117.

\bibitem{1994HarshmanLundy}
{\sc R.~A. Harshman and M.~E. Lundy}, {\em {P}arafac: {P}arallel factor
  analysis}, Comput. Stat. Data Anal.,  (1994), pp.~39--72.

\bibitem{Hitchcock}
{\sc F.~L. Hitchcock}, {\em The expression of a tensor or a polyadic as a sum
  of products}, J. Math. Phys., 6 (1927), pp.~164--189.

\bibitem{HornJohnson}
{\sc R.~A. Horn and C.~R. Johnson}, {\em {M}atrix {A}nalysis}, Cambridge
  University Press, 1990.

\bibitem{JiangSid2004}
{\sc T.~Jiang and N.~D. Sidiropoulos}, {\em {K}ruskal's {P}ermutation {L}emma
  and the {I}dentification of {CANDECOMP}/{PARAFAC} and {B}ilinear {M}odels
  with {C}onstant {M}odulus {C}onstraints}, IEEE Trans. Signal Process., 52
  (2004), pp.~2625--2636.

\bibitem{KoldaReview}
{\sc T.~G. Kolda and B.~W. Bader}, {\em {T}ensor {D}ecompositions and
  {A}pplications}, SIAM Review, 51 (2009), pp.~455--500.

\bibitem{Krijnen1991}
{\sc W.~P. Krijnen}, {\em The analysis of three-way arrays by constrained
  {P}arafac methods}, DSWO Press, Leiden, 1991.

\bibitem{Kroonenberg2008}
{\sc P.~M. Kroonenberg}, {\em {A}pplied {M}ultiway {D}ata {A}nalysis}, Hoboken,
  NJ: Wiley, 2008.

\bibitem{Kruskal1977}
{\sc J.~B. Kruskal}, {\em Three-way arrays: rank and uniqueness of trilinear
  decompositions, with application to arithmetic complexity and statistics},
  Linear Algebra Appl., 18 (1977), pp.~95--138.

\bibitem{Kruskal1989}
{\sc J.~B. Kruskal}, {\em Rank, decomposition, and uniqueness for 3-way and
  n-way arrays}, in Multiway Data Analysis, R. Coppi and S. Bolasco., eds.,
  Elsevier, North-Holland,  (1989), pp.~7--18.

\bibitem{Landsberg}
{\sc J.~M. Landsberg}, {\em {T}ensors: {G}eometry and {A}pplications}, AMS,
  Providence, Rhode Island, 2012.

\bibitem{LimCommon2010}
{\sc L.-H. Lim and P.~Comon}, {\em Multiarray signal processing: {T}ensor
  decomposition meets compressed sensing}, C.R. Mec., 338 (2010), pp.~311--320.

\bibitem{1988Topographic}
{\sc J.~M$\ddot{\text{o}}$cks}, {\em Topographic components model for
  event-related potentials and some biophysical considerations}, IEEE Trans.
  Biomed. Eng., 35 (1988), pp.~482--484.

\bibitem{Rhodes2010}
{\sc J.~A. Rhodes}, {\em A concise proof of {K}ruskal's theorem on tensor
  decomposition}, Linear Algebra Appl., 432 (2010), pp.~1818--1824.

\bibitem{CEM:CEM587}
{\sc N.~D. Sidiropoulos and R.~Bro}, {\em On the uniqueness of multilinear
  decomposition of {N}-way arrays}, J. Chemometrics, 14 (2000), pp.~229--239.

\bibitem{Senior00parallelfactor}
{\sc N.~D. Sidiropoulos, R.~Bro, and G.~B. Giannakis}, {\em {P}arallel {F}actor
  {A}nalysis in {S}ensor {A}rray {P}rocessing}, IEEE Trans. Signal Process., 48
  (2000), pp.~2377--2388.

\bibitem{890366}
{\sc N.~D. Sidiropoulos and L.~Xiangqian}, {\em Identifiability results for
  blind beamforming in incoherent multipath with small delay spread}, IEEE
  Trans. Signal Process., 49 (2001), pp.~228--236.

\bibitem{smilde2004multi}
{\sc A.K. Smilde, R.~Bro, and P.~Geladi}, {\em Multi-way analysis with
  applications in the chemical sciences}, J. Wiley, 2004.

\bibitem{Sorber}
{\sc L.~Sorber, M.~Van~Barel, and L.~De~Lathauwer}, {\em Optimization-based
  algorithms for tensor decompositions: canonical polyadic decomposition,
  decomposition in rank-(${L}_r$,${L}_r$,1) terms and a new generalization},
  ESAT-SISTA Internal Report, 12-37, Leuven, Belgium: Department of Electrical
  Engineering (ESAT), KU Leuven,  (2012).

\bibitem{Stegeman2009}
{\sc A.~Stegeman}, {\em On uniqueness conditions for {C}andecomp/{P}arafac and
  {I}ndscal with full column rank in one mode}, Linear Algebra Appl., 431
  (2009), pp.~211--227.

\bibitem{Stegeman2007}
{\sc A.~Stegeman and N.~D. Sidiropoulos}, {\em On {K}ruskal's uniqueness
  condition for the {C}andecomp/{P}arafac decomposition}, Linear Algebra Appl.,
  420 (2007), pp.~540--552.

\bibitem{Strassen1983}
{\sc V.~Strassen}, {\em Rank and optimal computation of generic tensors},
  Linear Algebra Appl., 52--53 (1983), pp.~645--685.

\bibitem{TenBerge2002}
{\sc J.~Ten~Berge and N.~D. Sidiropoulos}, {\em On uniqueness in
  {C}andecomp/{P}arafac}, Psychometrika, 67 (2002), pp.~399--409.

\bibitem{TenBerge2000}
{\sc J.~M.~F. Ten~Berge}, {\em The $k$-rank of a {K}hatri–--{R}ao product},
  Unpublished Note, Heijmans Institute of Psychological Research, University of
  Groningen, The Netherlands,  (2000).

\bibitem{LiuSid2001}
{\sc L.~Xiangqian and N.~D. Sidiropoulos}, {\em {C}ramer-{R}ao lower bounds for
  low-rank decomposition of multidimensional arrays}, IEEE Trans. Signal
  Process., 49 (2001), pp.~2074--2086.

\end{thebibliography}

\end{document}